\documentclass[reqno]{amsart}

\usepackage[T1]{fontenc}
\usepackage[utf8]{inputenc}
\usepackage{graphicx}
\usepackage{mathtools}
\usepackage{tikz}
\usetikzlibrary{arrows, backgrounds, calc, chains, decorations, patterns, positioning, shapes}

\usepackage{geometry}
\IfFormatAtLeastTF{2026-06-01}{}{\usepackage{thmtools}}
\usepackage[pdfpagelabels]{hyperref}
\usepackage{cleveref}
\usepackage{subcaption}
\usepackage{parskip}

\usepackage{amsfonts}
\usepackage{amsmath}
\usepackage{amsrefs}
\usepackage{amssymb}
\usepackage{amstext}
\usepackage{amsthm}

\usepackage{silence}
\theoremstyle{plain}
\newtheorem{theorem}{Theorem}[section]
\newtheorem{conjecture}[theorem]{Conjecture}

\newtheorem{lemma}[theorem]{Lemma}
\newtheorem{proposition}[theorem]{Proposition}

\theoremstyle{definition}
\newtheorem{definition}[theorem]{Definition}

\theoremstyle{remark}
\newtheorem{remark}[theorem]{Remark}
\newtheorem{example}[theorem]{Example}

\newcommand{\st}{^\mathsf{st}}

\renewcommand{\th}{^\mathsf{th}}

\DeclareMathOperator{\LD}{LD}

\DeclareMathOperator{\area}{area}
\DeclareMathOperator{\dinv}{dinv}
\DeclareMathOperator{\touch}{touch}

\newcommand{\N}{\mathbb{N}}

\newcommand{\Ht}{\widetilde{H}}

\newcommand{\dr}{\mathsf{dr}}
\newcommand{\dv}{\mathsf{dv}}
\newcommand{\mpe}{\mathsf{mp}}

\let\oldTheta\Theta
\let\Theta\undefined
\DeclareMathOperator{\Theta}{\oldTheta}

\makeatletter
\providecommand*{\shuffle}{%
  \mathbin{\mathpalette\shuffle@{}}%
}
\newcommand*{\shuffle@}[2]{%
  \sbox0{$#1\vcenter{}$}%
  \kern .15\ht0 % side bearing
  \rlap{\vrule height .25\ht0 depth 0pt width 2.5\ht0}%
  \raise.1\ht0\hbox to 2.5\ht0{%
    \vrule height 1.75\ht0 depth -.1\ht0 width .17\ht0 %
    \hfill
    \vrule height 1.75\ht0 depth -.1\ht0 width .17\ht0 %
    \hfill
    \vrule height 1.75\ht0 depth -.1\ht0 width .17\ht0 %
  }%
  \kern .15\ht0 % side bearing
}
\makeatother

\newcommand{\qbinom}[3][q]{\genfrac{[}{]}{0pt}{}{#2}{#3}_{#1}}

\newcommand{\<}{\langle}
\renewcommand{\>}{\rangle}

\makeatletter
\pgfkeys{
    /tikz/sharp angle/.code={%
        \pgfsetarrowoptions{sharp >}{#1}%
        \pgfsetarrowoptions{sharp <}{-#1}%
    },
    /tikz/sharp > angle/.code={%
        \pgfsetarrowoptions{sharp >}{#1}%
    },
    /tikz/sharp < angle/.code={%
        \pgfsetarrowoptions{sharp <}{#1}%
    },
    /tikz/sharp protrude/.code=\csname if#1\endcsname\qrr@tikz@sharp@z@-0.05\p@\else\qrr@tikz@sharp@z@\z@\fi,
    /tikz/sharp protrude/.default=true
}

\newdimen\qrr@tikz@sharp@z@
\qrr@tikz@sharp@z@\z@
\pgfarrowsdeclare{sharp >}{sharp >}{%
    \edef\pgf@marshal{\noexpand\pgfutil@in@{and}{\pgfgetarrowoptions{sharp >}}}%
    \pgf@marshal
    \ifpgfutil@in@
    \edef\pgf@tempa{\pgfgetarrowoptions{sharp >}}
    \expandafter\qrr@tikz@sharp@parse\pgf@tempa\@qrr@tikz@sharp@parse
    \else
    \qrr@tikz@sharp@parse\pgfgetarrowoptions{sharp >}and-\pgfgetarrowoptions{sharp >}\@qrr@tikz@sharp@parse
    \fi
    \pgfmathparse{max(\pgf@tempa,\pgf@tempb,0)}%
    \let\qrr@tikz@sharp@max\pgfmathresult
    \pgfmathsetlength\pgf@xa{.5*\pgflinewidth * tan(\qrr@tikz@sharp@max)}%
    \pgfarrowsleftextend{+\pgf@xa}%
    \pgfarrowsrightextend{+\pgf@xa}%
}{%
    \edef\pgf@marshal{\noexpand\pgfutil@in@{and}{\pgfgetarrowoptions{sharp >}}}%
    \pgf@marshal
    \ifpgfutil@in@
    \edef\pgf@tempa{\pgfgetarrowoptions{sharp >}}
    \expandafter\qrr@tikz@sharp@parse\pgf@tempa\@qrr@tikz@sharp@parse
    \else
    \qrr@tikz@sharp@parse\pgfgetarrowoptions{sharp >}and-\pgfgetarrowoptions{sharp >}\@qrr@tikz@sharp@parse
    \fi
    \pgfmathsetlength\pgf@ya{.5*\pgflinewidth * tan(max(\pgf@tempa,\pgf@tempb,0))}%
    \pgfmathsetlength\pgf@xa{-.5*\pgflinewidth * tan(\pgf@tempa)}%
    \pgfmathsetlength\pgf@xb{-.5*\pgflinewidth * tan(\pgf@tempb)}%
    \advance\pgf@xa\pgf@ya
    \advance\pgf@xb\pgf@ya
    \ifdim\pgf@xa>\pgf@xb
    \pgftransformyscale{-1}%
    \pgf@xc\pgf@xb
    \pgf@xb\pgf@xa
    \pgf@xa\pgf@xc
    \fi
    \pgfpathmoveto{\pgfqpoint{\qrr@tikz@sharp@z@}{.5\pgflinewidth}}%
    \pgfpathlineto{\pgfqpoint{\pgf@xa}{.5\pgflinewidth}}%
    \pgfpathlineto{\pgfqpoint{\pgf@ya}{+0pt}}%
    \pgfpathlineto{\pgfqpoint{\pgf@xb}{-.5\pgflinewidth}}%
    \pgfpathlineto{\pgfqpoint{\qrr@tikz@sharp@z@}{-.5\pgflinewidth}}%
    \pgfusepathqfill
}
\pgfarrowsdeclare{sharp <}{sharp <}{%
    \edef\pgf@marshal{\noexpand\pgfutil@in@{and}{\pgfgetarrowoptions{sharp <}}}%
    \pgf@marshal
    \ifpgfutil@in@
    \edef\pgf@tempa{\pgfgetarrowoptions{sharp <}}
    \expandafter\qrr@tikz@sharp@parse\pgf@tempa\@qrr@tikz@sharp@parse
    \else
    \expandafter\qrr@tikz@sharp@parse\pgfgetarrowoptions{sharp <}and-\pgfgetarrowoptions{sharp <}\@qrr@tikz@sharp@parse
    \fi
    \pgfmathparse{max(\pgf@tempa,\pgf@tempb,0)}%
    \let\qrr@tikz@sharp@max\pgfmathresult
    \pgfmathsetlength\pgf@xa{.5*\pgflinewidth * tan(\qrr@tikz@sharp@max)}%
    \pgfarrowsleftextend{+\pgf@xa}%
    \pgfarrowsrightextend{+\pgf@xa}%
}{%
    \edef\pgf@marshal{\noexpand\pgfutil@in@{and}{\pgfgetarrowoptions{sharp <}}}%
    \pgf@marshal
    \ifpgfutil@in@
    \edef\pgf@tempa{\pgfgetarrowoptions{sharp <}}
    \expandafter\qrr@tikz@sharp@parse\pgf@tempa\@qrr@tikz@sharp@parse
    \else
    \expandafter\qrr@tikz@sharp@parse\pgfgetarrowoptions{sharp <}and-\pgfgetarrowoptions{sharp <}\@qrr@tikz@sharp@parse
    \fi
    \pgfmathsetlength\pgf@ya{.5*\pgflinewidth * tan(max(\pgf@tempa,\pgf@tempb,0))}%
    \pgfmathsetlength\pgf@xa{-.5*\pgflinewidth * tan(\pgf@tempa)}%
    \pgfmathsetlength\pgf@xb{-.5*\pgflinewidth * tan(\pgf@tempb)}%
    \advance\pgf@xa\pgf@ya
    \advance\pgf@xb\pgf@ya
    \ifdim\pgf@xa>\pgf@xb
    \pgftransformyscale{-1}%
    \pgf@xc\pgf@xb
    \pgf@xb\pgf@xa
    \pgf@xa\pgf@xc
    \fi
    \pgfpathmoveto{\pgfqpoint{\qrr@tikz@sharp@z@}{.5\pgflinewidth}}%
    \pgfpathlineto{\pgfqpoint{\pgf@xa}{.5\pgflinewidth}}%
    \pgfpathlineto{\pgfqpoint{\pgf@ya}{+0pt}}%
    \pgfpathlineto{\pgfqpoint{\pgf@xb}{-.5\pgflinewidth}}%
    \pgfpathlineto{\pgfqpoint{\qrr@tikz@sharp@z@}{-.5\pgflinewidth}}%
    \pgfusepathqfill
}
\def\qrr@tikz@sharp@parse#1and#2\@qrr@tikz@sharp@parse{\def\pgf@tempa{#1}\def\pgf@tempb{#2}}
\makeatother

\title{The Theta Conjecture}

\author{Giovanni Interdonato}
\address{
    Classe di Scienze \newline \indent
    Scuola Normale Superiore \newline \indent
    Pisa, PI, 56126, Italia
}
\email{giovanni.interdonato@sns.it}

\author{Alessandro Iraci}
\address{
    Facoltà di Ingegneria e Informatica \newline \indent
    Università Pegaso \newline \indent
    Napoli, NA, 80143, Italia
}
\email{alessandro.iraci@unipegaso.it}

\begin{document}

\begin{abstract}
    We define a diagonal-inversion statistic on labeled Dyck paths carrying both decorated rises and decorated contractible valleys. This gives an explicit candidate for a bivariate refinement of the univariate Theta conjecture of D'Adderio, Iraci, and Vanden Wyngaerd, recently proved by D'Adderio, Pagaria, and the authors of this work. The univariate conjecture can be recovered from our bivariate version by setting $q=1$.

    The new conjecture recovers the rise and valley versions of the Delta conjecture when either decoration parameter vanishes, and provides a combinatorial interpretation of the symmetric function $\Theta_{e_l} \Theta_{e_k}\nabla e_{n-k-l}$, which is also conjectured to be the Frobenius characteristic of a certain graded module of diagonal coinvariants with two sets of commuting variables and two sets of anticommuting variables.
    
    In support of the Theta conjecture and its touching refinement, we prove its Schröder case, that is, the scalar product of the symmetric function side with $e_{n-d} h_d$ matches the combinatorial side restricted to Schröder paths. The proof uses a finer combinatorial argument than the previously known cases, leveraging two Gaussian product identities.
\end{abstract}

\maketitle
\tableofcontents

\section{Introduction}

Macdonald polynomials are a central family of symmetric functions introduced by Ian Macdonald \cite{Macdonald1988}. They depend on two parameters $q$ and $t$, whose rich structure connects algebraic combinatorics with representation theory and geometry \cites{Macdonald1995Book,Haiman1999MacdonaldPolynomialsGeometry}.

In the 1990s, Garsia and Haiman conjectured the Schur positivity of the (modified) Macdonald polynomials, claiming them to be the bigraded Frobenius characteristic of certain Garsia--Haiman modules \cite{GarsiaHaiman1993GradedRepresentationModel}. Their prediction was confirmed in 2001, when Haiman used the algebraic geometry of the Hilbert scheme to prove that the dimension of their modules is $n!$ \cite{Haiman2001nFactorial}, thus proving the $n!$ theorem.

In the course of these developments, it became clear that remarkable connections were to be found between Macdonald polynomial theory and the representation theory of the symmetric group. For example, during their quest for Macdonald positivity, Garsia and Haiman introduced the $\mathfrak{S}_n$-module of \emph{diagonal harmonics}, i.e.\ the coinvariants of the diagonal action of $\mathfrak{S}_n$ on polynomials in two sets of $n$ variables, and they conjectured that its Frobenius characteristic is given by $\nabla e_n$, where $\nabla$ is the \emph{nabla} operator on symmetric functions introduced in \cite{BergeronGarsiaHaimanTesler1999IdentitiesPositivityConjectures}, which acts diagonally on Macdonald polynomials. Haiman proved this conjecture in 2002 \cite{Haiman2002HilbertScheme}.

The positivity of $\nabla e_n$ led Haglund, Haiman, Loehr, Remmel, and Ulyanov to conjecture a combinatorial formula for it in terms of labeled Dyck paths weighted by the statistics $\dinv$ and $\area$, known as the \emph{shuffle conjecture} \cite{HaglundHaimanLoehrRemmelUlyanov2005ShuffleConjecture}. This conjecture was proved in 2018 by Carlsson and Mellit \cite{CarlssonMellit2018ShuffleConjecture}, and it has since inspired a wide range of research in the area.

The \emph{Delta conjecture}, first proposed in \cite{HaglundRemmelWilson2018DeltaConjecture}, extends the shuffle conjecture to the symmetric function $\Delta'_{e_{n-k-1}} e_n$, where $\Delta'_{f}$ is a certain operator on symmetric functions. The Delta conjecture predicts two different combinatorial formulas for this symmetric function, one in terms of decorated Dyck paths with $k$ decorated rises and another in terms of decorated Dyck paths with $k$ decorated contractible valleys. The rise version is now a theorem: D'Adderio and Mellit proved its compositional refinement \cite{DAdderioMellit2022CompositionalDelta}, while Blasiak, Haiman, Morse, Pun, and Seelinger independently proved the extended rise version \cite{BHMPS2023ProofExtendedDelta}. The full valley version remains open.

This result extends the connection with the representation theory of the symmetric group, as the symmetric function $\Delta'_{e_{n-k-1}} e_n$ is conjectured to be the Frobenius characteristic of a larger module of diagonal coinvariants with two sets of commuting variables and one set of anticommuting variables. This conjecture was first proposed by Zabrocki in \cite{Zabrocki2019ModuleDeltaConjecture}; the special case $t=0$, which corresponds to a module with one set of commuting variables and one set of anticommuting variables, was recently proved by Murai, Rhoades, and Wilson in \cite{MuraiRhoadesWilson2025Fields}.

The Theta operators introduced by D'Adderio, Iraci, and Vanden Wyngaerd \cite{DAdderioIraciVandenWyngaerd2021ThetaOperators} express the symmetric function side of the Delta conjecture as $\Theta_{e_k} \nabla e_{n-k}$ and suggest a common extension of the two combinatorial models. The expression $\Theta_{e_l} \Theta_{e_k}\nabla e_{n-k-l}$ is symmetric in $k$ and $l$, and it is conjectured to describe the component of exterior bidegree $(k,l)$ of diagonal coinvariants with two sets of commuting variables and two sets of anticommuting variables. This extends the module prediction of Zabrocki \cite{Zabrocki2019ModuleDeltaConjecture}. The purely fermionic specialization $q=t=0$, corresponding to two sets of anticommuting variables only, is known \cite{IraciRhoadesRomero2023FermionicTheta}.

In \cite[Conjecture~9.1]{DAdderioIraciVandenWyngaerd2021ThetaOperators}, a formula for $\left. \Theta_{e_l} \Theta_{e_k}\nabla e_{n-k-l} \right\rvert_{q=1}$ was proposed using labeled Dyck paths with $k$ decorated rises and $l$ decorated contractible valleys. This formula, including a refinement by the number of undecorated touches of the main diagonal, was proved in \cite[Theorems~6.1 and~6.21]{DIIP2026LeavingTheHall}.
At $t=0$, Iraci, Nadeau, and Vanden Wyngaerd obtained a different combinatorial formula through segmented Smirnov words, and transported its statistic to area-zero doubly decorated Dyck paths \cite[Section~4.5]{IraciNadeauVandenWyngaerd2024Smirnov}. Finding an explicit path statistic compatible with both kinds of decorations in positive area remained open.

In this paper we propose such a $\dinv$ statistic and formulate the corresponding bivariate Theta conjecture. Its new terms account for interactions between decorated rises and valleys, using a distinguished class of decorated valleys called \emph{star valleys}. The statistic was discovered with assistance from Aristotle \cite{Achim2025Aristotle}. When either decoration parameter is zero, the $\dinv$ reduces to the usual $\dinv$ appearing in the Delta conjecture. We also formulate a touching refinement with symmetric-function side $\Theta_{e_l}\Theta_{e_k}\nabla E_{n-k-l,r}$.
Notably, the area-zero restriction of our $\dinv$ does not coincide pointwise with the statistic transported from Smirnov words; proving the required equidistribution, and hence the new formula at $t=0$, remains open.

To support the Theta conjecture, we prove its Schröder case, that is, we verify that the scalar product of the symmetric function side with $e_{n-d} h_d$ matches the combinatorial side restricted to Schröder paths. This result is a generalization of the Schröder case of the Delta conjecture, proved in \cite{DAdderioIraciVandenWyngaerd2019GeneralizedDeltaSchroeder} for the rise version (now superseded by \cite{DAdderioMellit2022CompositionalDelta}, which proves the full rise version) and in \cite{DAdderioIraci2023} for the valley version. Our proof refines the known ones but also requires a finer combinatorial argument, leveraging two Gaussian product identities. We show via an explicit counterexample that a direct $q$-binomial interpretation of the various kinds of diagonal inversions does not hold.

In \Cref{sec:symmetric_functions} we collect the symmetric-function tools. In \Cref{sec:combinatorics} we define the statistic and state the conjectures. In \Cref{sec:schroeder} we prove the Schröder case. In \Cref{sec:ai} we describe the discovery procedure, computational evidence, and use of AI-assisted tools.

\section{Symmetric functions}
\label{sec:symmetric_functions}

We denote by $\Lambda$ the graded algebra of symmetric functions with coefficients in $\mathbb{Q}(q,t)$.

The standard bases of symmetric functions that will appear in our calculations are the monomial $\{m_\lambda\}_{\lambda}$, complete $\{h_{\lambda}\}_{\lambda}$, elementary $\{e_{\lambda}\}_{\lambda}$, power $\{p_{\lambda}\}_{\lambda}$, and Schur $\{s_{\lambda}\}_{\lambda}$ bases.

We denote by $\<\, , \>$ the \emph{Hall scalar product} on $\Lambda$, defined by \[ \< p_\lambda, p_\mu \> = z_\lambda \delta_{\lambda,\mu}, \quad \text{ where } \quad z_\lambda = \prod_{i \geq 1} i^{m_i(\lambda)} m_i(\lambda)! \] and $m_i(\lambda)$ is the multiplicity of $i$ in $\lambda$. With respect to this scalar product, the bases $\{h_\lambda\}_\lambda$ and $\{m_\lambda\}_\lambda$ are dual, and the Schur basis $\{s_\lambda\}_\lambda$ is orthonormal.

For $f \in \Lambda$, we denote by $f^\perp$ the operator adjoint to multiplication by $f$ with respect to the Hall scalar product, that is, for every $g, h \in \Lambda$, we have $\< f^\perp g, h \> = \< g, fh \>$.

When relevant, we will use the convention $e_0 = h_0 = 1$ and $e_k = h_k = 0$ for $k < 0$.

\subsection{\texorpdfstring{$q$}{q}-analogs}
\label{ssec:q-analogs}

It is convenient to introduce the so-called $q$-notation. In general, a $q$-analog of an expression is a generalization involving a parameter $q$ that reduces to the original one for $q \rightarrow 1$.

\begin{definition}
	For a natural number $n \in \mathbb{N}$, we define its $q$-analog as \[ [n]_q \coloneqq \frac{1-q^n}{1-q} = 1 + q + q^2 + \dots + q^{n-1}. \]
\end{definition}

Given this definition, one can define the $q$-factorial and the $q$-binomial as follows.

\begin{definition}
	We define \[ [n]_q! \coloneqq \prod_{k=1}^{n} [k]_q \quad \text{and} \quad \qbinom{n}{k} \coloneqq \frac{[n]_q!}{[k]_q![n-k]_q!} \]
\end{definition}

\begin{definition}
	For any variable $z$ and $n \in \N \cup \{ \infty \}$, we define the \emph{$q$-Pochhammer symbol} as \[ (z;q)_n \coloneqq \prod_{j=0}^{n-1} (1-zq^j) = (1-z) (1-zq) (1-zq^2) \cdots (1-zq^{n-1}). \]
\end{definition}

Notice that the identity \[\qbinom{n}{k}=\frac{\left(q^{n+1-k};q\right)_k}{\left(q;q\right)_k}\]
allows us to define the $q$-binomial for general integers $n,k$ with $k\ge0$. In particular we have that $\displaystyle\qbinom{n}{k}=0$ if $k>n\ge0$ and that $\displaystyle\qbinom{-1}{0}=1$.

We recall some well-known identities regarding $q$-binomials. Proofs can be found, for example, in \cite[Section~3.2]{Sag20}.

\begin{theorem}[$q$-Pascal identities]
    For $n \geq k \geq 0$,
    \begin{equation}
        \qbinom{n+1}{k+1}= q^{k+1} \qbinom{n}{k+1} +\qbinom{n}{k} = \qbinom{n}{k+1} + q^{n-k} \qbinom{n}{k}. \label{eq:q_pascal}
    \end{equation}
\end{theorem}

\begin{theorem}[$q$-binomial Theorem]\label{thm:q-binomial}
    For $n \geq 0$,
    \[ \prod_{k=0}^{n-1}(1 + z q^k) = \sum_{k=0}^{n} q^{\binom{k}{2}} \qbinom{n}{k} z^k. \]
\end{theorem}

We also recall one identity from \cite{DAdderioIraci2023}.

\begin{proposition}[{\cite[Lemma~4.8]{DAdderioIraci2023}}]
    \label{prop:q_sum}
    For $r, p, b \geq 0$,
    \[ q^{\binom{p}{2}} \qbinom{r}{p} \qbinom{r+b-1}{b} = \sum_{c=0}^{b} q^{\binom{p+c}{2}} \qbinom{r}{p+c} q^{\binom{c}{2}} \qbinom{p+c}{c} \qbinom{p+b-1}{b-c}. \]
\end{proposition}

\subsection{Plethysm and Macdonald polynomials}

We will make extensive use of the \emph{plethystic notation} (cf.\ \cite[Chapter~1]{Haglund2008Book}): for a symmetric function $f$ and an expression $E$ in the variables $x_1, x_2, \dots$, we denote by $f[E]$ the symmetric function obtained by substituting the power sum $p_k$ with $p_k[E]$, where $p_k[E]$ is obtained by substituting each variable $x_i$ with $x_i^k$ in $E$. For example, if $E = x_1 + x_2$, then $p_2[E] = x_1^2 + x_2^2$.

Using this notation, we can define the symmetric functions $E_{n,k}$ via the expansion \[ e_n \left[ X \frac{1-z}{1-q} \right] = \sum_{k=0}^n \frac{(z;q)_k}{(q;q)_k} E_{n,k}. \] 
When relevant, we will use the convention $E_{0,0} = 1$ and $E_{n,k} = 0$ unless $1 \leq k \leq n$.

For a partition $\mu \vdash n$, we denote by $\Ht_\mu[X; q,t]$ the \emph{(modified) Macdonald polynomials}, defined as the unique symmetric functions satisfying the triangularity and normalization axioms
		\begin{align*}
			(\text{T1}) & \qquad \Ht_\mu[X(1-q); q,t] = \sum_{\lambda \geq \mu} a_{\lambda \mu}(q,t) s_\lambda[X] \\
			(\text{T2}) & \qquad \Ht_\mu[X(1-t); q,t] = \sum_{\lambda \geq \mu'} b_{\lambda \mu}(q,t) s_\lambda[X] \\
			(\text{N})  & \qquad \< \Ht_\mu[X;q,t], s_{(n)}[X] \> = 1.
		\end{align*}
        
Macdonald polynomials form a basis of the algebra of symmetric functions $\Lambda$, which is a modification of the one introduced by Macdonald \cite{Macdonald1995Book}. It is orthogonal with respect to the \emph{star scalar product} $\langle \cdot, \cdot \rangle_\ast$, defined by \[ \< p_\lambda, p_\mu \>_\ast = (-1)^{|\mu| - \ell(\mu)} \prod_{i=1}^{\ell(\mu)} (1-q^{\mu_i})(1-t^{\mu_i}) z_\mu \delta_{\lambda,\mu}, \] or equivalently, by \[ \< f, g \>_\ast = \< \omega f[MX], g \>, \] where $\omega$ is the standard involution on $\Lambda$ defined by $\omega e_n = h_n$, and $M = (1-q)(1-t)$. We denote by $f^\ast = f[X/M]$ the inverse of the plethystic substitution $f \mapsto f[MX]$.

With respect to the star scalar product, the Macdonald polynomials are orthogonal, and we have \[ \< \Ht_\lambda, \Ht_\mu \>_\ast = w_\mu(q,t) \delta_{\lambda,\mu}, \] where \[ w_\mu(q,t) = \prod_{c \in \mu} (q^{a_\mu(c)} - t^{l_\mu(c)+1})(t^{l_\mu(c)} - q^{a_\mu(c)+1}). \]

Next, recall the definition of the Pieri coefficients $d^{f}_{\mu,\nu}$: for any symmetric function $f \in \Lambda$, \[ f \widetilde{H}_{\nu} = \sum_{\mu} d^{f}_{\mu,\nu} \widetilde{H}_\mu. \]

Finally, recall the Cauchy--Macdonald identity, which follows from the classical Cauchy identity applied to the Macdonald basis: \[ e_n\left[\frac{XY}{M}\right] = \sum_{\mu \vdash n} \frac{\Ht_\mu[X] \Ht_\mu[Y]}{w_\mu}. \]

\subsection{Macdonald eigenoperators}

If we identify the partition $\mu$ with its Young diagram, i.e.\ with the collection of cells $\{(i,j)\mid 1\leq i\leq \mu_j, 1\leq j\leq \ell(\mu)\}$, then for each cell $c\in \mu$ we define the \emph{arm}, \emph{leg}, \emph{co-arm}, and \emph{co-leg} (denoted, respectively, by $a_\mu(c), l_\mu(c), a_\mu'(c), l_\mu'(c)$) to be the numbers of cells in $\mu$ that are strictly to the right, below, to the left, and above $c$ in $\mu$, respectively (see \Cref{fig:notation}).

\begin{figure}[htbp]
    \centering
    \begin{tikzpicture}[scale=0.4]
        \draw[gray,opacity=.6](0,0) grid (15,10);
        \fill[white] (1,-0.1)|-(3,1) |- (6,3) |- (9,7) |- (15.1,8) |- (1,-.1);
        \fill[blue, opacity=.15] (0,7) rectangle (9,8) (3,3) rectangle (4,10);
        \fill[blue, opacity=.5] (3,7) rectangle (4,8);
        \draw (6,7.5) node {\tiny{Arm}} (3.5,5) node[rotate=90] {\tiny{Leg}} (3.5, 9) node[rotate = 90] {\tiny{Co-leg}} (1.5,7.5) node {\tiny{Co-arm}} ;
    \end{tikzpicture}
    \caption{Arm, leg, co-arm, and co-leg of a cell of a partition.}
    \label{fig:notation}
\end{figure}
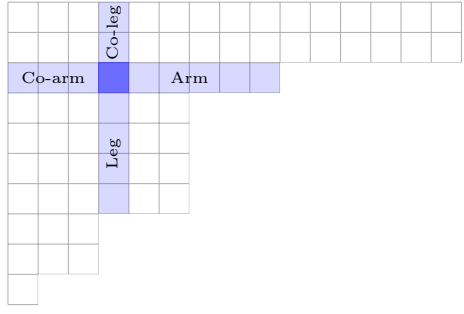

For every partition $\mu$, we define the following constants:

\[ B_{\mu} \coloneqq B_{\mu}(q,t) = \sum_{c \in \mu} q^{a_{\mu}'(c)} t^{l_{\mu}'(c)}, \qquad
    \Pi_{\mu} \coloneqq \Pi_{\mu}(q,t) = \prod_{c \in \mu / (1)} (1-q^{a_{\mu}'(c)} t^{l_{\mu}'(c)}). \]

Notice that in the definition of $\Pi_\mu$, the exponents $a_{\mu}'$ and $l_{\mu}'$ are evaluated with respect to the shape $\mu$, not $\mu / (1)$; the latter notation only indicates that the cell $(1,1)$ should be omitted from the product (otherwise the product would be $0$).

We recall the definitions of the following linear operators on $\Lambda$.

\begin{definition}[{\cite[Equation~(3.11)]{BergeronGarsia1999ScienceFiction}}]
    \label{def:nabla}
    We define the linear operator $\nabla \colon \Lambda \rightarrow \Lambda$ on the eigenbasis of Macdonald polynomials as \[ \nabla \Ht_\mu = e_{\lvert \mu \rvert}[B_\mu] \Ht_\mu. \]
\end{definition}

\begin{definition}
    \label{def:pi}
    We define the linear operator $\mathbf{\Pi} \colon \Lambda \rightarrow \Lambda$ on the eigenbasis of Macdonald polynomials as \[ \mathbf{\Pi} \Ht_\mu = \Pi_\mu \Ht_\mu \] where we conventionally set $\Pi_{\varnothing} \coloneqq 1$.
\end{definition}

\begin{definition}
    \label{def:delta}
    For $f \in \Lambda$, we define the linear operators $\Delta_f, \Delta'_f \colon \Lambda \rightarrow \Lambda$ on the eigenbasis of Macdonald polynomials as \[ \Delta_f \Ht_\mu = f[B_\mu] \Ht_\mu, \qquad \qquad \Delta'_f \Ht_\mu = f[B_\mu-1] \Ht_\mu. \]
\end{definition}

Observe that, on the vector space of homogeneous symmetric functions of degree $n$, denoted by $\Lambda^{(n)}$, the operator $\nabla$ equals $\Delta_{e_n}$. Also note that \[ \mathbf{\Pi} = \sum_{n=0}^\infty (-1)^n \Delta'_{e_n} \] on symmetric functions with no constant term.

\begin{definition}[{\cite[Equation~(28)]{DAdderioIraciVandenWyngaerd2021ThetaOperators}}]
    \label{def:theta}
    For any symmetric function $f \in \Lambda^{(n)}$, we define the \emph{Theta operators} on $\Lambda$ as follows: for every $F \in \Lambda^{(m)}$ we set
    \begin{equation*}
        \Theta_f F \coloneqq
        \left\{\begin{array}{ll}
            0 & \text{if } n \geq 1 \text{ and } m=0 \\
            f \cdot F & \text{if } n=0 \text{ and } m=0 \\
            \mathbf{\Pi} f \left[\frac{X}{M}\right] \mathbf{\Pi}^{-1} F & \text{otherwise}
        \end{array}
        \right. ,
    \end{equation*}
    and we extend the definition by linearity to all $f, F \in \Lambda$.
\end{definition}

\section{The combinatorics of the Theta conjecture}
\label{sec:combinatorics}

We begin by giving precise definitions for the combinatorics of the Delta and Theta conjectures.

\begin{definition}
    A \emph{Dyck path} of size $n$ is a lattice path starting at $(0,0)$, ending at $(n,n)$, using only unit up (vertical) steps and right (horizontal) steps, and staying weakly above the line $x=y$. A \emph{labeled Dyck path} is a Dyck path $\pi$ together with a labeling $w \colon [n] \to \mathbb{Z}_+$ of its vertical steps, such that consecutive vertical steps have strictly increasing labels (from bottom to top). We will draw the labels of the vertical steps in the square immediately to the right of each step.

    A \emph{rise} of a labeled Dyck path is a vertical step that is preceded by another vertical step.

    A \emph{valley} of a labeled Dyck path is a vertical step $v$ preceded by at least one horizontal step. A valley $v$ is \emph{contractible} if it is preceded by either two horizontal steps or a horizontal step that is itself preceded by a vertical step whose label is strictly smaller than $v$'s label.

    A \emph{decorated labeled Dyck path} $\pi$ is a labeled Dyck path together with a choice of rises and contractible valleys to be \emph{decorated}.
    We set
    \begin{align*}
         & \dr(\pi) = \{i \in [n] \mid \text{the $i\th$ vertical step of $\pi$ is a decorated rise}\}     \\
         & \dv(\pi) = \{i \in [n] \mid \text{the $i\th$ vertical step of $\pi$ is a decorated valley}\}.
    \end{align*}
    We decorate rises with a $\ast$ and valleys with a $\bullet$, and these decorations are displayed in the square to the left of the vertical step. The set of decorated labeled Dyck paths of size $n$ with $k$ decorated rises and $l$ decorated valleys is denoted by $\LD(n)^{\ast k, \bullet l}$.
\end{definition}

See \Cref{fig:path-example} for an example of an element of $\LD(8)^{\ast 2, \bullet 2}$.

\begin{figure}[htbp]
    \centering
    \begin{tikzpicture}[scale=.72]
        \draw[step=1.0, gray!60, thin] (0,0) grid (8,8);
        \draw[gray!60, thin] (0,0) -- (8,8);
        \draw[blue!60, line width = 1.6 pt] (0,0) -- (0,1) -- (0,2) -- (1,2) -- (1,3) -- (2,3) -- (3,3) -- (3,4) -- (3,5) -- (3,6) -- (4,6) -- (5,6) -- (5,7) -- (6,7) -- (6,8) -- (7,8) -- (8,8);
        \node at (0.5,0.5) {$2$};
        \draw (0.5,0.5) circle (.4cm);
        \node at (0.5,1.5) {$3$};
        \draw (0.5,1.5) circle (.4cm);
        \node at (1.5,2.5) {$4$};
        \draw (1.5,2.5) circle (.4cm);
        \node at (3.5,3.5) {$1$};
        \draw (3.5,3.5) circle (.4cm);
        \node at (3.5,4.5) {$2$};
        \draw (3.5,4.5) circle (.4cm);
        \node at (3.5,5.5) {$4$};
        \draw (3.5,5.5) circle (.4cm);
        \node at (5.5,6.5) {$3$};
        \draw (5.5,6.5) circle (.4cm);
        \node at (6.5,7.5) {$2$};
        \draw (6.5,7.5) circle (.4cm);
        \node at (1-1-0.5,1+0.5) {$\ast$};
        \node at (6-3-0.5,5+0.5) {$\ast$};
        \node at (2-1-0.5,2+0.5) {$\bullet$};
        \node at (7-2-0.5,6+0.5) {$\bullet$};
    \end{tikzpicture}
    \caption{An element of $\LD(8)^{\ast 2, \bullet 2}$.}
    \label{fig:path-example}
\end{figure}

\begin{definition}[Monomial]
    \label{def:monomial}
    For a decorated labeled Dyck path $\pi$, we define the monomial $x^\pi = \prod_{i=1}^n x_{w_i}$, where $w_i$ is the label of the $i\th$ vertical step of $\pi$.
\end{definition}

\begin{definition}[Area]
    \label{def:area}
    Given a decorated labeled Dyck path $\pi$ of size $n$, its \emph{area word} is the word of nonnegative integers whose $i\th$ letter equals the number of unit squares between the $i\th$ vertical step of the path and the line $x=y$. If $a$ is the area word of $\pi$, the \emph{area} of $\pi$ is \[\area(\pi) \coloneqq \sum_{i \in [n] \setminus \dr(\pi)} a_i.\]
\end{definition}

For example, the decorated labeled Dyck path in \Cref{fig:path-example} has area $4$.

\begin{remark}[{\cite[Remark~1.19]{DAdderioIraciVandenWyngaerd2022TheBible}}]
	\label{rmk:rises-falls-correspondence}
	Note that, if we call a horizontal step followed by another horizontal step a \emph{fall}, then there is a natural bijection between rises and falls.
    Indeed, the joining point of the two vertical steps of a rise is a point where a path $\pi$ vertically crosses a certain diagonal parallel to the main diagonal.
    Since the path must end at the main diagonal, it must cross the same diagonal horizontally at least once, via a fall. We map the rise to the first such fall; this yields a bijection (see \Cref{fig:rises-falls-correspondence}).
    Therefore, we might equivalently decorate falls instead of the corresponding rises.
\end{remark}

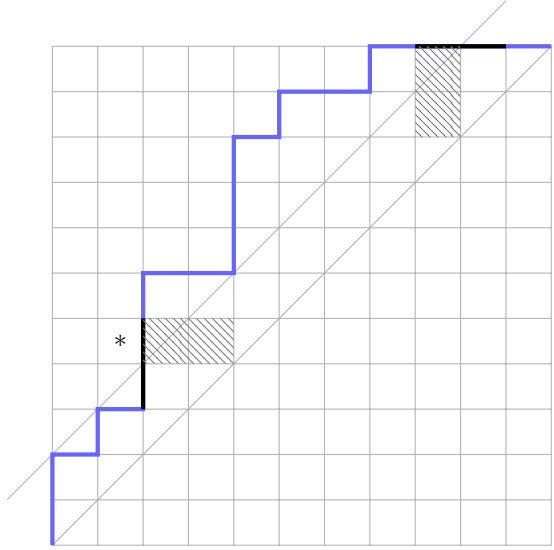
\begin{figure}[htbp]
	\centering
	\begin{tikzpicture}[scale=0.6]
		\draw[gray!60, thin] (0,0) grid (11,11) (0,0) -- (11,11) (-1,1) -- (10,12);
		\draw[blue!60, ultra thick] (0,0) -- (0,2) -- (1,2) -- (1,3) -- (2,3) -- (2,6) -- (4,6) -- (4,9) -- (5,9) -- (5,10) -- (7,10) -- (7,11) -- (11,11);
		\draw[ultra thick] (2,3) -- (2,5) (8,11) -- (10,11);
		\fill[pattern=north west lines, pattern color=gray] (2,4) rectangle (4,5) (8,11) rectangle (9,9);
		\draw  (1.5,4.5) node {$\ast$};
	\end{tikzpicture} 
	\caption{Correspondence between rises and falls.}
	\label{fig:rises-falls-correspondence}
\end{figure}

Before defining the $\dinv$ statistic, which is the main novelty of this paper, we need to introduce some additional notation.

\begin{definition}
    A decorated valley of a decorated labeled Dyck path $\pi$ is a \emph{star valley} if it is preceded by two horizontal steps, and the first of these is a decorated fall, via the bijection described in \Cref{rmk:rises-falls-correspondence}. We denote the set of star valleys of $\pi$ by $\dv^\ast(\pi)$.
\end{definition}

We write $i \stackrel{\ast}{\rightarrow} j$, and say that $i$ \emph{star-attacks} $j$, if $i < j$, $a_i = a_j + 1$, $i \in \dr$, and $j \in \dv$.

\begin{definition}[Dinv]
    \label{def:dinv}
    Given a decorated labeled Dyck path $\pi$ of size $n$, its \emph{diagonal inversions} are the pairs of indices $i < j$ such that:
    \begin{enumerate}
        \item $i \not \in \dv$, $a_i = a_j$, and $w_i < w_j$; or \label{enum:dinv_1}
        \item $i \not \in \dv$, $a_i = a_j + 1$, and $w_i > w_j$; or \label{enum:dinv_2}
        \item $i \in \dv^\ast$, $a_i = a_j$, $j \in \dv$, and $w_i < w_j$; or \label{enum:dinv_3}
        \item $i \in \dv^\ast$, $a_i = a_j$, $j \not \in \dv$, and $w_i > w_j$. \label{enum:dinv_4}
    \end{enumerate}

    The \emph{dinv} of $\pi$ is the number
    \[ \dinv(\pi) = \# \{ i < j \mid (i,j) \text{ is a diagonal inversion of } \pi \} - \# (\dv \setminus \dv^\ast) - \# \{ i < j \mid i \stackrel{\ast}{\rightarrow} j \}. \]
\end{definition}

The first two terms are classical, while the last two terms are new. Notice that, for $j \in \dv^\ast$, the set $\{ i \mid i \stackrel{\ast}{\rightarrow} j \}$ is nonempty, and it contains the index of the decorated rise corresponding to the decorated fall preceding $j$. The discovery of this statistic has been assisted by Aristotle \cite{Achim2025Aristotle}.

The decorated labeled Dyck path in \Cref{fig:path-example} has dinv $4$: the pairs $(2,3), (5,7)$ contribute to the first term, $(2,4), (6,7), (6,8)$ contribute to the second term, none contribute to the third term, and $(7,8)$ contributes to the fourth term; then we have one non-star valley at index $3$, and $6 \stackrel{\ast}{\rightarrow} 7$.

\begin{proposition}
    The dinv statistic is nonnegative.
\end{proposition}

\begin{proof}
    Let $j \in \dv$. We prove the stronger statement
    \begin{equation}
        \label{eq:local-nonnegativity}
        \#\{ i < j \mid (i,j) \text{ is a diagonal inversion } \} \geq \# \{k \mid k \stackrel{\ast}{\rightarrow} j \}
            + \delta_{j \notin \dv^\ast}.
    \end{equation}

    Let $s < j$ be an index such that $a_s = a_j$ and $w_s < w_j$; if $s \notin (\dv \setminus \dv^\ast)$ then $(s,j)$ is an inversion of type~\eqref{enum:dinv_1} if $s \notin \dv$ and of type~\eqref{enum:dinv_3} if $s \in \dv^\ast$. Likewise, a rise $r < j$ with $a_r = a_j + 1$ and $w_r > w_j$ gives an inversion of type~\eqref{enum:dinv_2}.

    We want to associate each star-attack relation $k \stackrel{\ast}{\rightarrow} j$ with a distinct inversion $(s,j)$ of type~\eqref{enum:dinv_1}, \eqref{enum:dinv_2}, or \eqref{enum:dinv_3}, and, if $j \notin \dv^\ast$, we want to associate $j$ with another distinct inversion $(s,j)$ of the same types.

    We define a backward search that, given any index $s < j$ such that $a_s = a_j$ and $w_s < w_j$, finds another index $s'< j$ with the same properties and $s' \notin (\dv \setminus \dv^\ast)$, or a rise $r < j$ with $a_r = a_j + 1$ and $w_r > w_j$. We proceed as follows.
    \begin{itemize}
        \item If $s \notin (\dv \setminus \dv^\ast)$, stop at $s$.
        \item If $s \in \dv \setminus \dv^\ast$, then $s$ is a decorated valley that is not a star valley. 
        \begin{itemize}
            \item If $a_{s-1} = a_s$, then $w_{s-1} < w_s < w_j$ by the contractibility condition. Replace $s$ by $s-1$ and continue.
            \item If $a_{s-1} > a_s$, then $s$ is preceded by two horizontal steps, that is, a fall. Let $r$ be the rise matched with that fall. Since $s$ is not a star valley, $r$ is not decorated.
            \begin{itemize}
                \item If $w_r > w_j$, stop at $r$.
                \item If $w_r \leq w_j$, then $w_{r-1} < w_r \leq w_j$ as $r$ is a rise. Replace $s$ by $r-1$ and continue.
            \end{itemize}
        \end{itemize}
    \end{itemize}
    Since every step of the search moves strictly to the left, it must terminate, ultimately at the first step on diagonal $a_j$ if it has not stopped earlier. That first step is not a valley and is not in $(\dv \setminus \dv^\ast)$. The output of the search determines one of the inversions described above.

    For each star-attack relation $k \stackrel{\ast}{\rightarrow} j$, if $w_k > w_j$, we assign the inversion $(k,j)$ of type~\eqref{enum:dinv_2}. If $w_k \leq w_j$, then $w_{k-1} < w_k \leq w_j$, so we can start the backward search at $s = k-1$; we assign the resulting inversion to the star-attack relation $k \stackrel{\ast}{\rightarrow} j$.
    
    If $j \notin \dv^\ast$, we need to find another distinct inversion. When $a_{j-1} = a_j$, we start its search at $j-1$, using $w_{j-1} < w_j$.  When $a_{j-1} > a_j$, since $j$ is not a star valley, it is preceded by a fall whose matched rise $r$ is not decorated; we assign $(r,j)$ if $w_r > w_j$, and otherwise we start the search at $r-1$.
    
    After ordering the attacking decorated rises as $k_1<\cdots<k_m$, let $v_h$ be the valley reached by the fall matched with $k_h$. Before the search for $k_h$ could pass to the left of $k_{h-1}$, it must reach $v_{h-1}$. This index is a stopping point: if it is not decorated it is outside $\dv$, while if it is decorated it is a star valley because its matched rise $k_{h-1}$ is decorated. Therefore the search for $k_h$ stays strictly to the right of $k_{h-1}$, whereas the search for $k_{h-1}$ stays to its left. The search for $j$ similarly stops at or after $v_m$; equality $v_m=j$ would make $j$ a star valley, a contradiction. This proves that all assigned inversions are distinct, and the thesis now follows.
\end{proof}

We recall the statement of the Delta conjecture, which was first proposed in \cite[Conjecture~1.1]{HaglundRemmelWilson2018DeltaConjecture}.

\begin{conjecture}
    \label{conj:delta}
    For $n, k \geq 0$, we have
    \[ \Delta'_{e_{n-k-1}} e_{n} = \sum_{\pi \in \LD(n)^{\ast k, \bullet 0}} q^{\dinv(\pi)} t^{\area(\pi)} x^\pi = \sum_{\pi \in \LD(n)^{\ast 0, \bullet k}} q^{\dinv(\pi)} t^{\area(\pi)} x^\pi. \]
\end{conjecture}

The first equality is the rise version (now a theorem, first proved in \cite{DAdderioMellit2022CompositionalDelta} and then independently in \cite{BHMPS2023ProofExtendedDelta}), while the equality with the second sum is known as the valley version (still open).

Recall that $\Delta'_{e_{n-k-1}} e_n = \Theta_{e_k} \nabla e_{n-k}$ \cite[Theorem~3.1]{DAdderioIraciVandenWyngaerd2021ThetaOperators}, so the LHS of the Delta conjecture can be rewritten as $\Theta_{e_k} \nabla e_{n-k}$.

We can now state the \emph{Theta conjecture}, extending the original Theta conjecture \cite[Conjecture~9.1]{DAdderioIraciVandenWyngaerd2021ThetaOperators} to include a $q$-statistic.

\begin{conjecture}[Theta conjecture]
    \label{conj:theta}
    For $n, k, l \geq 0$, we have
    \[ \Theta_{e_l} \Theta_{e_k} \nabla e_{n-k-l} = \sum_{\pi \in \LD(n)^{\ast k, \bullet l}} q^{\dinv(\pi)} t^{\area(\pi)} x^\pi. \]
\end{conjecture}

This conjecture unifies the rise and valley versions of the Delta conjecture \cite[Conjecture~1.1]{HaglundRemmelWilson2018DeltaConjecture}, which are obtained by setting $l=0$ and $k=0$, respectively. It has been tested for $n \leq 10$ using SageMath \cite{SageMath}.

We also have a touching refinement, but the naïve compositional refinement does not match the expected symmetric function side (cf. \cites{HaglundMorseZabrocki2012CompositionalShuffleConjecture,DAdderioIraciVandenWyngaerd2021ThetaOperators}).

\begin{definition}[Touch number]
    The \emph{touch number} of a path $\pi \in \LD(n)^{\ast k, \bullet l}$ is the number of nondecorated vertical steps that touch the main diagonal. We denote it by $\touch(\pi)$, and we let $\LD(n,r)^{\ast k, \bullet l}$ be the subset of $\LD(n)^{\ast k, \bullet l}$ of paths with touch number $r$.
\end{definition}

Notice that $1 \leq \touch(\pi) \leq n-k-l$. For example, the path in \Cref{fig:path-example} has touch number $2$.

\begin{conjecture}[Theta conjecture, touching refinement]
    \label{conj:theta-touch}
    For $n, k, l, r \geq 0$, we have
    \[ \Theta_{e_l} \Theta_{e_k} \nabla E_{n-k-l,r} = \sum_{\pi \in \LD(n,r)^{\ast k, \bullet l}} q^{\dinv(\pi)} t^{\area(\pi)} x^\pi. \]    
\end{conjecture}

The special case when $q=1$ is a touching refinement of the original Theta conjecture, proved in \cite[Theorem~6.21]{DIIP2026LeavingTheHall}.

\begin{remark}
    \label{rmk:reading-order}
    Usually, the reading word of a decorated labeled Dyck path $\pi$ with respect to the \emph{dinv} statistic is the word obtained by reading the labels of the vertical steps from the bottom-most diagonal to the top-most diagonal, and from left to right within each diagonal.
    
    For these objects, in order to obtain the maximum number of diagonal inversions when the reading word is the identity, we need to slightly modify this order: we still read the labels from the bottom-most diagonal to the top-most diagonal, but we read the labels of the decorated valleys after the labels of the other steps. In other words, set
    \[ i \prec j \iff (a_i, \delta_{i\in\dv}, i) <_{\mathrm{lex}} (a_j, \delta_{j\in\dv}, j). \]
\end{remark}

\begin{definition}[Standardization]
    \label{def:standardization}
    For $\pi \in \LD(n,r)^{\ast k, \bullet l}$, we define its \emph{standardization} $\mathsf{std}(\pi)$ to be the decorated labeled Dyck path obtained by replacing the smallest label of $\pi$ with $1$, the second smallest with $2$, and so on, breaking ties by reverse $\prec$-order (meaning that if $w_i = w_j$ and $i \prec j$, $i$ gets the larger label).
    
    We say that $\pi$ is \emph{standard} if its set of labels is exactly $[n]$. We denote by $\LD(n)^{\ast k, \bullet l}_{\mathrm{std}}$ the subset of standard decorated labeled Dyck paths in $\LD(n)^{\ast k, \bullet l}$.
\end{definition}

\begin{lemma}
    \label{lem:standardization}
    Standardizing preserves column strictness, contractibility of decorated valleys, and $\dinv$.
\end{lemma}

\begin{proof}
    For each of the four types of diagonal inversion, its label condition can be written $w_u < w_v$ with $u \prec v$.
    In particular, unequal labels retain their relative order under standardization, and equal labels become decreasing in $\prec$-order, and therefore create no new diagonal inversion. Star-valley status and the correction terms depend
    only on the path and its decorations.

    The strict inequalities required in a column, or at a decorated valley preceded by exactly one horizontal step, also run forward in $\prec$. They are preserved because their labels are unequal. The remaining contractibility conditions depend only on the path.
\end{proof}

Using \Cref{lem:standardization}, we can rewrite the Theta conjecture as a sum over standard decorated labeled Dyck paths associated with a fundamental quasisymmetric function.

\begin{conjecture}[Theta conjecture, standard version]
    \label{conj:theta-standard}
    For $n, k, l \geq 0$, we have
    \[ \Theta_{e_l} \Theta_{e_k} \nabla e_{n-k-l} = \sum_{\pi \in \LD(n)^{\ast k, \bullet l}_{\mathrm{std}}} q^{\dinv(\pi)} t^{\area(\pi)} F_{\mathsf{ides}(\sigma(\pi))}, \]
    where $\sigma(\pi)$ is the permutation obtained by reading the labels of $\mathsf{std}(\pi)$ in reverse $\prec$-order, $\mathsf{ides}(\sigma)$ is the descent set of $\sigma^{-1}$, and $F$ is Gessel's fundamental quasisymmetric function, where the correspondence between subsets of $[n-1]$ and compositions of $n$ can be made explicit by the identification $\mathsf{ides}(\sigma) = \{i \mid \operatorname{pos}(i) < \operatorname{pos}(i+1)\}$.
\end{conjecture}

It is easy to see that Conjecture~\ref{conj:theta-standard} is equivalent to Conjecture~\ref{conj:theta} by grouping the terms of the sum according to their standardization.

\section{The Schröder case}
\label{sec:schroeder}

Classically, the Schröder case of a shuffle-like conjecture is obtained by taking the scalar product with $e_{n-d} h_d$, which fixes the reading word to be a shuffle of $1, \dots, n-d$ (in increasing reading order) and $n, \dots, n-d+1$ (in decreasing reading order). Here, the reading order is the one described in \Cref{rmk:reading-order}. Notice that this procedure does not require the combinatorial generating function to be a symmetric function, as the scalar product with $e_{n-d} h_d$ can be extended to the space of quasisymmetric functions as the linear functional that selects the sets $[n-d-1]$ and $[n-d]$ in the fundamental basis.

In this case, the objects carry no explicit labels, and the selection of the $d$ rows with large labels is equivalent to choosing a subset of marked \emph{peaks}, that is, vertical steps followed by horizontal steps, with the additional condition that a marked peak cannot be immediately followed by a vertical step that is a decorated valley.

Indeed, suppose that a marked peak is immediately followed by a decorated valley. Since the latter is preceded by exactly one horizontal step, its contractibility would require the label of the vertical step of the marked peak to be smaller than the label of the valley. This is impossible in the Schröder specialization: the marked peak has one of the large labels and, by definition, its label is larger than every label that comes after it in the reading order, including that of the decorated valley. Conversely, this is the only additional restriction introduced by forgetting the labels, since a decorated valley preceded by at least two horizontal steps is automatically contractible.

Let us denote by $\LD(n,r)^{\ast k, \bullet l}_{\circ d}$ the set of decorated Dyck paths of size $n$ with $k$ decorated rises, $l$ decorated valleys, touch number $r$, and $d$ marked peaks, with the condition that no marked peak is immediately followed by a decorated valley. We set \[
\mpe(\pi) = \{i \in [n] \mid \text{the $i\th$ vertical step of $\pi$ is in a marked peak}\} .
\]

The definition of $\dinv$, taking into account the restriction on the order of labels, is adjusted accordingly.

\begin{definition}[Dinv]
    \label{def:dinv_mp}
    Given $\pi\in\LD(n,r)^{\ast k, \bullet l}_{\circ d}$ its \emph{diagonal inversions} are the pairs of indices $i < j$ such that:
    \begin{enumerate}
        \item $i \not \in \dv$, $a_i = a_j$, and $i \not \in \mpe$; or \label{enum:ndinv_1}
        \item $i \not \in \dv$, $a_i = a_j + 1$, and $j \not \in \mpe$; or \label{enum:ndinv_2}
        \item $i \in \dv^\ast$, $a_i = a_j$, $j \in \dv$, and $i \not \in \mpe$; or \label{enum:ndinv_3}
        \item $i \in \dv^\ast$, $a_i = a_j$, $j \not \in \dv$, and $j \not \in \mpe$. \label{enum:ndinv_4}
    \end{enumerate}

    The \emph{dinv} of $\pi$ is the number
    \[ \dinv(\pi) = \# \{ i < j \mid (i,j) \text{ is a diagonal inversion of } \pi \} - \# (\dv \setminus \dv^\ast) - \# \{ i < j \mid i \stackrel{\ast}{\rightarrow} j \}. \]
\end{definition}

The goal of this section is to prove the following theorem.

\begin{theorem}
    \label{thm:schroder}
    For $n, k, l, r, d \geq 0$, we have
    \[ \langle \Theta_{e_l} \Theta_{e_k} \nabla E_{n-k-l,r}, e_{n-d} h_d \rangle = \sum_{\pi \in \LD(n,r)^{\ast k, \bullet l}_{\circ d}} q^{\dinv(\pi)} t^{\area(\pi)}, \]
    where as usual, symmetric functions and combinatorial families with parameters outside their natural ranges are understood to be zero.
\end{theorem}

\begin{proof}
    Let
    \[ \LD_{q,t}(n,r)^{\ast k, \bullet l}_{\circ d} = \sum_{\pi \in \LD(n,r)^{\ast k, \bullet l}_{\circ d}} q^{\dinv(\pi)} t^{\area(\pi)} \]
    and 
    \[ S(n,r)^{\ast k, \bullet l}_{\circ d} = \langle \Theta_{e_l} \Theta_{e_k} \nabla E_{n-k-l,r}, e_{n-d} h_d \rangle \]
    be the RHS and the LHS of \Cref{thm:schroder}, respectively. By \Cref{prop:algebraic_recursion} and \Cref{prop:combinatorial_recursion}, both families of polynomials satisfy the same recursion and initial condition.
    
    For $n>0$, the case $r=0$ is zero on both sides. For $r\geq1$, every recursive term has size $n-r-b<n$, so the result follows by induction.
\end{proof}

\subsection{Algebraic recursion}

The goal of this subsection is to prove the following proposition, which gives a recursion for $S(n,r)^{\ast k, \bullet l}_{\circ d}$.

\begin{proposition}
    \label{prop:algebraic_recursion}
    For $n, k, l, r, d \geq 0$, we have
    \begin{align*}
        S(n,r)^{\ast k, \bullet l}_{\circ d} & = t^{n-r-k-l} \sum_{b=0}^{l} \sum_{p=0}^{r} t^{l-b}  q^{\binom{p}{2}} \qbinom{r}{p} \qbinom{r+b-1}{b} \\
        & \quad \times \sum_{u=0}^{n-r-k-l} \sum_{v=0}^{p+b} q^{\binom{v}{2}} \qbinom{p+b}{v} \qbinom{p+b+u-1}{u} S(n-r-b, u+v)^{\ast k-v, \bullet l-b}_{\circ d-(r-p)},
    \end{align*}
    with the initial condition $S(0,r)^{\ast k, \bullet l}_{\circ d} = \delta_{r,0} \delta_{k,0} \delta_{l,0} \delta_{d,0}$.
\end{proposition}

We are going to prove that this family of polynomials is, up to a change of variables, the family
\[F^{d,l}_{n-k,r,k} \coloneqq t^{n-r-k-l} \langle \Delta_{h_{n-r-k-l}} \Delta_{e_l} e_{n-d} \left[ X [r]_q \right], e_k h_{n-k-d} \rangle \]
considered in \cite[Equation~(28)]{DAdderioIraciVandenWyngaerd2019GeneralizedDeltaSchroeder}, which is known to satisfy the same recursion: the one given in \cite[Theorem~3.4]{DAdderioIraciVandenWyngaerd2019GeneralizedDeltaSchroeder} (cf.\ \cite[Theorem~5.1]{Iraci2019PhD}), iterated twice. We start with a preliminary lemma.

\begin{lemma}
    \label{lem:theta_inner_product}
    Let $0 \leq d, r \leq n$ and let $g \in \Lambda$ be homogeneous of degree $n - r$. Then
    \[ \langle \Theta_g \widetilde{H}_{(r)}, e_{n-d} h_d \rangle = \left(\Delta_g e_{n-d}^* \right)[(1-t)(1-q^r)]. \]
\end{lemma}

\begin{proof}
    If $r=0$, the identity reduces to $0=0$ unless $n=d=0$, in which case it reduces to $1=1$.

    Assume $r \geq 1$. By definition of $\Theta_g$ and the Pieri coefficients, we have
    \[ \Theta_g \widetilde{H}_{(r)} = \mathbf{\Pi} \left( g^* \Pi_{(r)}^{-1} \widetilde{H}_{(r)} \right) = \Pi_{(r)}^{-1} \sum_{\mu \vdash n} d^{g^*}_{\mu,(r)} \Pi_\mu \widetilde{H}_\mu, \]
    so pairing with $e_{n-d} h_d$ and using the hook coefficient identity $\langle \widetilde{H}_\mu, e_{n-d} h_d \rangle = e_{n-d}[B_\mu]$, we obtain
    \[ \langle \Theta_g \widetilde{H}_{(r)}, e_{n-d} h_d \rangle = \Pi_{(r)}^{-1} \sum_{\mu \vdash n} \Pi_\mu e_{n-d}[B_\mu] d^{g^*}_{\mu,(r)}. \]
    By Haglund's summation formula \cite[Theorem~2.6]{Haglund2004Schroeder} (cf.\ \cite[Equation~(110)]{DAdderioIraciVandenWyngaerd2021ThetaOperators}), with $A = g^*$, $F = e_{n-d}^*$, and $\nu = (r)$,
    \[ \sum_{\mu \vdash n} \Pi_\mu e_{n-d}[B_\mu] d^{g^*}_{\mu,(r)} = \Pi_{(r)} \bigl(\Delta_g e_{n-d}^*\bigr)[MB_{(r)}]. \]
    Since $B_{(r)} = [r]_q$, we have $MB_{(r)} = (1-q)(1-t)[r]_q = (1-t)(1-q^r)$, and the claim follows.
\end{proof}

\begin{lemma}
    \label{lem:S-identity}
    For $n, k, l, r, d \geq 0$, we have
    \[ S(n,r)^{\ast k, \bullet l}_{\circ d} = t^{n-r-k-l} \langle \Delta_{h_{n-r-k-l}} \Delta_{e_l} e_{n-d} \left[ X [r]_q \right], e_k h_{n-k-d} \rangle. \]
\end{lemma}

\begin{proof}
    By definition of $S(n,r)^{\ast k, \bullet l}_{\circ d}$, we have
    \[ S(n,r)^{\ast k, \bullet l}_{\circ d} = \langle \Theta_{e_l} \Theta_{e_k} \nabla E_{n-k-l,r}, e_{n-d} h_d \rangle. \]
    
    Setting $m = n-r-k-l$, we apply Haglund's formula \cite[Theorem~2.5]{Haglund2004Schroeder} and \cite[Remark~8.4]{DAdderioIraciVandenWyngaerd2021ThetaOperators} to get
    \[ \nabla E_{m+r,r} = t^m \Theta_{h_m} \widetilde{H}_{(r)}[X], \]
    and by the multiplicativity of $\Theta$ operators, we obtain
    \[ S(n,r)^{\ast k, \bullet l}_{\circ d} = t^m \langle \Theta_{e_k e_l h_m} \widetilde{H}_{(r)}, e_{n-d} h_d \rangle. \]
    
    Now we apply \Cref{lem:theta_inner_product} with $g = e_k e_l h_m$, resulting in
    \[ \langle \Theta_g \widetilde{H}_{(r)}, e_{n-d} h_d \rangle = \bigl(\Delta_g e_{n-d}^* \bigr)[(1-t)(1-q^r)]; \]
    by the Macdonald--Cauchy identity, $e_{n-d}^* = \sum_{\mu \vdash n-d} \widetilde{H}_\mu / w_\mu$, so
    \[  S(n,r)^{\ast k, \bullet l}_{\circ d} = t^m \sum_{\mu \vdash n-d} \frac{e_k[B_\mu] e_l[B_\mu] h_m[B_\mu]}{w_\mu} \widetilde{H}_\mu[(1-t)(1-q^r)]. \]

    Using the hook-coefficient identity $\langle \widetilde{H}_\mu, e_k h_{n-k-d} \rangle = e_k[B_\mu]$, we can rewrite the above as
    \begin{align*}
        S(n,r)^{\ast k, \bullet l}_{\circ d} & = t^m \sum_{\mu \vdash n-d} h_m[B_\mu] e_l[B_\mu] \frac{\widetilde{H}_\mu[(1-t)(1-q^r)]}{w_\mu} \langle \widetilde{H}_\mu, e_k h_{n-k-d} \rangle \\
        & = t^m \left\langle \Delta_{h_m} \Delta_{e_l} \sum_{\mu \vdash n-d} \frac{\widetilde{H}_\mu[(1-t)(1-q^r)] \widetilde{H}_\mu[X]}{w_\mu}, e_k h_{n-k-d} \right\rangle.
    \end{align*}    
    On the other hand, by the Macdonald--Cauchy identity applied with $Y = (1-t)(1-q^r)$,
    \[ e_{n-d}\left[X [r]_q \right] = \sum_{\mu \vdash n-d} \frac{\widetilde{H}_\mu[(1-t)(1-q^r)] \widetilde{H}_\mu[X]}{w_\mu}, \]
    so
    \[ S(n,r)^{\ast k, \bullet l}_{\circ d} = t^m \langle \Delta_{h_m} \Delta_{e_l} e_{n-d}\left[X [r]_q \right], e_k h_{n-k-d} \rangle, \]
    as desired.
\end{proof}

\begin{proof}[Proof of \Cref{prop:algebraic_recursion}]
    Set $m=n-r-k-l$. By commutativity of the Theta operators and \Cref{lem:S-identity}, we have 
    \[ S(n,r)^{\ast k,\bullet l}_{\circ d} = F_{n-l,r;l}^{(d,k)}. \]
    Applying \cite[Theorem~3.4]{DAdderioIraciVandenWyngaerd2019GeneralizedDeltaSchroeder} once gives
    \[
        S(n,r)^{\ast k,\bullet l}_{\circ d} = t^m \sum_{b=0}^{l} \sum_{p=0}^{r} q^{\binom{p}{2}} \qbinom{r}{p} \qbinom{r+b-1}{b} F_{n-d,p+b;m}^{(n-d-k,n-l-d-p)}.
    \]
    If $F$ is now described by an initial condition, the result follows by an elementary computation. Otherwise, we can iterate the recursion to obtain
    \begin{align*}
        F_{n-d,p+b;m}^{(n-d-k,n-l-d-p)} & = t^{l-b}\sum_{u=0}^{m} \sum_{v=0}^{p+b} q^{\binom{v}{2}} \qbinom{p+b}{v}\qbinom{p+b+u-1}{u} F_{n-r-l,u+v;l-b}^{(d-r+p,k-v)}.
    \end{align*}
    By \Cref{lem:S-identity}, we have
    \[ F_{n-r-l,u+v;l-b}^{(d-r+p,k-v)} = S(n-r-b,u+v)^{\ast(k-v),\bullet(l-b)}_{\circ(d-r+p)}, \]
    and putting the pieces together yields \Cref{prop:algebraic_recursion}.
\end{proof}

\subsection{Combinatorial recursion}
The goal of this subsection is to prove the following proposition, which gives the same recursion as \Cref{prop:algebraic_recursion} for $\LD_{q,t}(n,r)^{\ast k, \bullet l}_{\circ d}$.

\begin{proposition}
    \label{prop:combinatorial_recursion}
    For $n, k, l, r, d \geq 0$, we have
    \begin{align*}
        \LD_{q,t}(n,r)^{\ast k, \bullet l}_{\circ d} & = t^{n-r-k-l} \sum_{b=0}^{l} \sum_{p=0}^{r} t^{l-b}  q^{\binom{p}{2}} \qbinom{r}{p} \qbinom{r+b-1}{b} \\
        & \quad \times \sum_{u=0}^{n-r-k-l} \sum_{v=0}^{p+b} q^{\binom{v}{2}} \qbinom{p+b}{v} \qbinom{p+b+u-1}{u} \LD_{q,t}(n-r-b, u+v)^{\ast k-v, \bullet l-b}_{\circ d-(r-p)},
    \end{align*}
    with the initial condition $\LD_{q,t}(0,r)^{\ast k, \bullet l}_{\circ d} = \delta_{r,0} \delta_{k,0} \delta_{l,0} \delta_{d,0}$.
\end{proposition}

\begin{proof}
    First, we restate our recursion to match \cite[Theorem~4.11]{DAdderioIraci2023}, that is, the corresponding statement for $k=0$. Using \Cref{prop:q_sum}, we can rewrite the recursion as
    \begin{align*}
        \LD_{q,t}(n,r)^{\ast k, \bullet l}_{\circ d} & = t^{n-r-k-l} \sum_{b=0}^{l} \sum_{h=0}^{b} \sum_{p=0}^{r-h} t^{l-b}  q^{\binom{h}{2}} \qbinom{p+h}{h} q^{\binom{p+h}{2}} \qbinom{r}{p+h} \qbinom{p+b-1}{b-h} \\
        & \quad \times \sum_{u=0}^{n-r-k-l} \sum_{v=0}^{p+b} q^{\binom{v}{2}} \qbinom{p+b}{v} \qbinom{p+b+u-1}{u} \LD_{q,t}(n-r-b, u+v)^{\ast k-v, \bullet l-b}_{\circ d-(r-p)},
    \end{align*}
    which, up to the change of variables $b \mapsto j$, $h \mapsto u$, $p \mapsto v$, $u \mapsto s$, recovers \cite[Theorem~4.11]{DAdderioIraci2023} when $k=0$ (implying $v=0$).

    We prove this equivalent form of the recursion. The base case $n=0$ is clear, as $\LD(0,r)^{\ast k, \bullet l}_{\circ d}$ is empty unless $r=k=l=d=0$.

    Given $\pi \in \LD(n,r)^{\ast k, \bullet l}_{\circ d}$, we construct a smaller Dyck path $\bar\pi$ as follows. For each vertical step starting on the main diagonal, we delete that step together with the horizontal step immediately preceding the next return to the main diagonal.
    
    We want to check that $\bar\pi$ satisfies the compatibility condition between marked peaks and decorated valleys. First, every surviving marked peak is still a peak. Indeed, if the horizontal step of a marked peak were one of the deleted steps, then that step ends on the main diagonal, forcing also the vertical step of the peak to start on the main diagonal and hence to be deleted as well. Therefore a marked peak either disappears entirely or survives as a marked peak in $\bar\pi$.

    Suppose now that a marked peak were immediately followed by a decorated valley in $\bar\pi$. By assumption, they cannot be consecutive in $\pi$, so some steps that touch the main diagonal must lie between them in $\pi$. This implies that in $\pi$ the decorated valley is preceded by a deleted vertical step. Hence it is not a valley in $\pi$, but a rise, a contradiction.
        
    We keep track of the following parameters:

    \begin{itemize}
        \item $b$ is the number of decorated valleys on the main diagonal,
        \item $r-p$ is the number of marked peaks on the main diagonal,
        \item $h$ is the number of those marked peaks that are also decorated valleys,
        \item $u$ is the number of vertical steps on the first diagonal that are not decorated,
        \item $v$ is the number of vertical steps on the first diagonal that are decorated rises.
    \end{itemize}
    With these parameters, we have that \[ \bar\pi \in \LD(n-r-b, u+v)^{\ast (k-v), \bullet (l-b)}_{\circ (d-(r-p))}. \]
    We can also derive these (nonnegative) quantities:

    \begin{itemize}
        \item $r+b$ is the total number of vertical steps on the main diagonal,
        \item $p+b$ is the number of vertical steps on the main diagonal that are not marked peaks,
        \item $b-h$ is the number of decorated valleys on the main diagonal that are not marked peaks,
        \item $p+h$ is the number of vertical steps on the main diagonal with no decorations or marks,
        \item $r-p-h$ is the number of marked peaks on the main diagonal that are not decorated valleys.
    \end{itemize}

    Since $\touch(\bar\pi) = u+v$, we have $u+v$ nondecorated vertical steps on the main diagonal of $\bar\pi$, each of them corresponding to a subpath of $\bar\pi$ starting with that step and ending with the next (nondecorated) vertical step on the main diagonal (or the end of the path). A preimage of $\bar\pi$ is obtained by distributing these subpaths among the $p+b$ deleted vertical steps on the main diagonal that are not marked peaks, in their original order, and assigning decorations to the first step of the first subpath inserted in $v$ of these slots.

    We first determine the change in area. After the deletion, every surviving vertical step is shifted down by one diagonal. There are $n-r-b$ such steps: on the main diagonal, we have $r$ steps that are not decorated valleys, and $b$ steps that are decorated valleys. The vertical steps carrying the $k$ rise decorations all survive the deletion and do not contribute to $\area(\pi)$; therefore each of the remaining $n-r-b-k$ surviving steps decreases its area contribution by one, and hence
    \[ \area(\pi) = \area(\bar\pi) + n-r-b-k, \]
    which gives the factor $t^{n-r-k-l} t^{l-b} t^{\area(\bar\pi)}$. 

    All contributions to $\dinv$ involving only surviving indices are unchanged. The surviving entries of the area word all decrease by one, and their relative order and peak marks are preserved. The two steps immediately preceding a surviving decorated valley lie within the same subpath. If they form a fall, its matched rise also lies in the subpath, and the relevant decoration is preserved. Thus the star-valley status is unchanged. Finally, a rise that loses its decoration is such that its star-attack relations involve only deleted valleys on the main diagonal. Each of the subpaths just described contributes exactly one surviving step in the next diagonal that is not a decorated valley. If there are $j$ unmarked deleted steps to its right, it contributes $j$ diagonal inversions of type \eqref{enum:ndinv_2} with the main diagonal.
    
    In order to understand the remaining contributions to $\dinv$, we encode the deleted vertical steps of $\pi$ in a word $W$ of length $r+b$ with letters $N,V,P,D$, corresponding to a deleted vertical step carrying, respectively, neither decoration nor mark, only a valley decoration, only a peak mark, or both. The word $W$ starts with $N$ or $P$, and every $V$ or $D$ follows an $N$ or $V$. These conditions are also sufficient for the reconstructed path to satisfy the required compatibility.
    
    By construction, the number of letters $N$ is $s \coloneqq p+h$, the number of letters $V$ is $b-h$, the number of letters $P$ is $r-p-h$, and the number of letters $D$ is $h$. The positions in which we may insert subpaths are those corresponding to letters $N$ or $V$, since we cannot go up from a peak as it would destroy the peak itself. We have $p+h+b-h=p+b$ such positions; we will also need to choose $v$ of them to receive restored rise decorations.
        
    If $p+b=0$, all the steps are marked peaks, $\bar\pi$ is empty, and $u=v=0$; the required contribution is $1$. In what follows, assume $p+b>0$. The word $W$ can be written uniquely in the form
    \[
        P^{b_0} N V^{c_1} D^{\epsilon_1} \cdots P^{b_{s-1}} N V^{c_s} D^{\epsilon_s} P^{b_{s}},
    \]
    with $\sum_i b_i = r-p-h$, $\sum_i c_i = b-h$, $\sum_i \epsilon_i = h$, and $\epsilon_i \in \{0,1\}$ for all $i$.

    The factor $q^{\binom{p+h}{2}}$ counts the diagonal inversions of type \eqref{enum:ndinv_1} between the $p+h$ vertical steps on the main diagonal with no decorations or marks (that is, the $N$ letters) among themselves. The factor \[ \qbinom{r}{p+h} \] counts the diagonal inversions of type \eqref{enum:ndinv_1} between the $p+h$ vertical steps on the main diagonal with no decorations or marks (the $N$ letters) and the $r-(p+h)$ marked peaks on the main diagonal that are not decorated valleys (the $P$ letters): we have a diagonal inversion whenever one of the former appears to the left of one of the latter.
    
    We now insert the subpaths. The path $\bar\pi$ has $u+v$ ordered touch-subpaths. Given a deleted-step word $W$, its $p+b$ letters of type $N$ or $V$ determine the insertion slots. For each slot $j$, choose $u_j \geq 0$ and $\delta_j \in \{0,1\}$, with
    \[ \sum_j u_j=u,\qquad \sum_j\delta_j=v. \]
    In the order of the slots, distribute the touch-subpaths into consecutive groups of sizes $u_j+\delta_j$. If $\delta_j=1$, decorate the first vertical step of the first subpath in that slot. Its $u_j$ remaining subpaths begin with nondecorated steps. If $\delta_j=0$, all $u_j$ first steps are not decorated. Only the first subpath of a slot can begin with a restored rise, because each later subpath is preceded by a horizontal step. These data determine the preimage uniquely.

    Index the slots from right to left by $j=0,\ldots,p+b-1$. The ordinary insertions contribute $\sum_j j u_j$ to $\dinv$, independently of the binary choices $\delta_j$. Hence their generating function is 
    \[ \sum_{u_0+\cdots+u_{p+b-1}=u}q^{\sum_j j u_j} = \qbinom{p+b+u-1}{u}. \]
    Indeed, this binomial keeps track of the dinv of type \eqref{enum:ndinv_2} between the $u$ nondecorated vertical steps on the first diagonal and the $p+b$ steps on the main diagonal that are not marked peaks (the $N$ and $V$ letters): we have a diagonal inversion whenever one of the former appears to the left of one of the latter. The $-1$ in the $q$-binomial accounts for the fact that the first step of the path lies to the left of all the $u$ steps on the first diagonal, and thus cannot contribute to any diagonal inversions of type \eqref{enum:ndinv_2}.    
    
    If we disregard decorations on the rises, and pretend there is no star valley on the main diagonal, the remaining contribution to $\dinv$ is given by the number of pairs $(i,j)$ with $i<j$ such that $W_i=N$ and $W_j=V$ or $D$, minus the number of restored decorated valleys (which is $b$). This is equal to
    \[ \sum_{i=1}^s i c_i + \sum_{i=1}^s i \epsilon_i - b = \sum_{i=1}^s (i-1) (c_i + \epsilon_i), \]
    where the equality holds because $\sum_{i=1}^s (c_i + \epsilon_i) = (b-h) + h = b$.

    We now restore the $v$ subpaths that receive a rise decoration. In addition to the type~\eqref{enum:ndinv_2} inversions contributed by the first step of the subpath receiving the decoration, we must subtract one for every decorated main-diagonal valley to its right, with which it forms a star-attack relation. If the next deleted step is a decorated valley, it becomes a star valley: its former penalty of $-1$ disappears, and its new inversions contribute one for each later $N$, as well as one for each later $V$ or $D$ if this star valley is unmarked. Notice that transforming a decorated valley into a star valley does not lose any diagonal inversion.

    Recall that at most one of the restored first-diagonal rises can occur in each slot, so we insert these $v$ subpaths in $v$ different positions among the $p+b$ available ones in $W$, after any $N$ or $V$ letter in $W$. We distinguish two cases: we insert the subpath after the last $V$ of a block of the form $N V^{c_i} D^{\epsilon_i} P^{b_i}$ (which means after $N$ if $c_i = 0$), or before one of its $V$'s. Set
    \[ c_{>i} = \sum_{j>i} c_j, \quad \epsilon_{>i} = \sum_{j>i} \epsilon_j \]
    and the analogous quantities $c_{\geq i} = c_{>i} + c_i$ and $\epsilon_{\geq i} = \epsilon_{>i} + \epsilon_i$.

    In the first case, we obtain $(s-i) + c_{>i}$ inversions of type \eqref{enum:ndinv_2} with the $s-i$ letters $N$ and the $c_{>i}$ letters $V$ after the inserted subpath. Each of the $c_{>i} + \epsilon_{\geq i}$ decorated valleys to its right contributes a correction of $-1$ to the exponent, since each forms a star-attack relation with this decorated rise. If $\epsilon_i=1$, the $D$ ending the block becomes a star valley, and its contribution of $-1$ disappears, since it no longer belongs to $\dv \setminus \dv^\ast$. It also forms an inversion pair of type \eqref{enum:ndinv_4} with each of the $s-i$ letters $N$ to its right. Thus the total exponent is
    \[
        (s-i)+c_{>i} - \left(c_{>i}+\epsilon_{\geq i}\right) + \epsilon_i(1+s-i) = (1+\epsilon_i)(s-i) - \epsilon_{> i}.
    \]

    In the second case, suppose that we insert the subpath after the $j\th$ letter $V$ of the block $N V^{c_i} D^{\epsilon_i} P^{b_i}$, with $0 \leq j < c_i$. Then we obtain $(s-i) + c_{>i} + (c_i - j)$ inversions of type \eqref{enum:ndinv_2} with the $s-i$ letters $N$ and the $c_{>i} + (c_i - j) = c_{\geq i} - j$ letters $V$ after the inserted subpath. Each of the $c_{\geq i} - j + \epsilon_{\geq i}$ decorated valleys to its right contributes a correction of $-1$ to the exponent, since each forms a star-attack relation with this decorated rise. The $(j+1)\st$ letter $V$ of the block becomes a star valley, so its contribution of $-1$ disappears, since it no longer belongs to $\dv \setminus \dv^\ast$.  It also forms an inversion pair of type \eqref{enum:ndinv_3} with each of the $c_{\geq i} - j + \epsilon_{\geq i} - 1$ decorated valleys to its right, and an inversion pair of type \eqref{enum:ndinv_4} with each of the $s-i$ letters $N$ to its right. Thus the total exponent is
    \[
        (s-i) + c_{\geq i} - j - \left(c_{\geq i}-j+\epsilon_{\geq i} \right) + 1 + (c_{\geq i}-j+\epsilon_{\geq i} - 1) + (s-i) = 2(s-i) + c_{\geq i} - j.
    \]
    
    We use an auxiliary variable $z$ to record the number of restored rise decorations, so that we seek the coefficient of $z^v$. Set
    \begin{align*}
        T_{s,h}(z) & \coloneqq \sum_{\substack{\epsilon_i\in\{0,1\} \\ \sum_{i=1}^s\epsilon_i=h}}
        q^{\sum_{i=1}^s(i-1)\epsilon_i} \prod_{i=1}^s \left(1+zq^{(1+\epsilon_i)(s-i)-\epsilon_{>i}}\right) \\
        U_{s,c}(z) & \coloneqq \sum_{\substack{c_i\geq0\\\sum_{i=1}^s c_i=c}}
        q^{\sum_{i=1}^s(i-1)c_i} \prod_{i=1}^s\prod_{j=0}^{c_i-1} \left(1+zq^{2(s-i) + c_{\geq i} - j}\right).
    \end{align*}

    Since we have $v$ restored rise decorations, the contribution to $\dinv$ of the restored subpaths is given by the coefficient of $z^v$ in $T_{s,h}(z)U_{s,c}(z)$, where $s=p+h$ and $c=b-h$.

    We use the following two identities:
    \begin{lemma}
        \label{lem:TU_formulas}
        For $s \geq h \geq 0$ and $c \geq 0$, we have
        \begin{align}
            T_{s,h}(z) & = q^{\binom{h}{2}}\qbinom{s}{h} \prod_{j=0}^{s-1}(1+zq^j), \label{eq:T_formula} \\
            U_{s,c}(z) & = \qbinom{s+c-1}{c} \prod_{j=s}^{s+c-1}(1+zq^j). \label{eq:U_formula}
        \end{align}
    \end{lemma}

    \begin{proof}
        We will prove the first statement by induction on $s$, and the second by induction on $s+c$.

        We have that $T_{0,0}(z)=1$ and $T_{s,h}(z)=0$ for $h<0$ or $h>s$. Moreover $U_{s,0}(z)=1$ and $U_{0,c}(z)=0$ for $c>0$. These provide the base cases for the inductions

        Separating the cases $\epsilon_1=0$ and $\epsilon_1=1$, for $s\ge1$ and $0\le h\le s$ we have 
        \begin{align*}
            T_{s,h}(z) & = q^h(1+zq^{s-1-h}) T_{s-1,h}(z)+q^{h-1}(1+zq^{2s-h-1}) T_{s-1,h-1}(z) \\
            & = q^{\binom{h}{2}} \prod_{j=0}^{s-2}(1+zq^j)
                \left(q^h(1+zq^{s-1-h})\qbinom{s-1}{h}+(1+zq^{2s-h-1})\qbinom{s-1}{h-1}\right) \\
            & = q^{\binom{h}{2}} \prod_{j=0}^{s-2}(1+zq^j) 
                \left(q^h\qbinom{s-1}{h} + \qbinom{s-1}{h-1} + 
                    zq^{s-1}\left(\qbinom{s-1}{h} +q^{s-h}\qbinom{s-1}{h-1}\right) \right) \\
            & = q^{\binom{h}{2}} \prod_{j=0}^{s-2}(1+zq^j) \left(\qbinom{s}{h}+zq^{s-1}\qbinom{s}{h}\right) \\
            & = q^{\binom{h}{2}} \qbinom{s}{h} \prod_{j=0}^{s-1}(1+zq^j),
        \end{align*}
        where the second equality follows from the induction hypothesis
        and the fourth from the $q$-Pascal identities in \Cref{eq:q_pascal}. This proves \Cref{eq:T_formula}.

        Similarly, separating the cases $c_1=0$ and $c_1>0$, we obtain, for
        $s\geq 1$ and $c\geq 1$,
        \begin{align*}
            U_{s,c}(z) & = q^c U_{s-1,c}(z) + \left(1+zq^{2s+c-2}\right) U_{s,c-1}(z) \\
            & = \prod_{j=s}^{s+c-2}(1+zq^j)
            \left( q^c(1+zq^{s-1})\qbinom{s+c-2}{c} +\left(1+zq^{2s+c-2}\right)\qbinom{s+c-2}{c-1} \right) \\
            & = \prod_{j=s}^{s+c-2}(1+zq^j)
            \left[ q^c\qbinom{s+c-2}{c} + \qbinom{s+c-2}{c-1} \right. \\
            & \quad \left. {} + zq^{s+c-1} \left( \qbinom{s+c-2}{c} + q^{s-1} \qbinom{s+c-2}{c-1} \right) \right] \\
            & = \prod_{j=s}^{s+c-2}(1+zq^j) \left(1+zq^{s+c-1}\right) \qbinom{s+c-1}{c} \\
            & = \qbinom{s+c-1}{c} \prod_{j=s}^{s+c-1}(1+zq^j).
        \end{align*}
        Indeed, when $c_1=0$, deleting the initial zero shifts the indices and contributes the factor $q^c$.  When $c_1>0$, replacing $c_1$ by $c_1-1$ yields the second term of the recurrence.
        The second equality above then follows from the induction hypothesis, while the fourth follows from the two $q$-Pascal identities in \Cref{eq:q_pascal}. This proves \Cref{eq:U_formula}.
    \end{proof}

    Since $s+c=p+b$ and by \Cref{lem:TU_formulas}, the contribution to $\dinv$ is
    \[ T_{s,h}(z)U_{s,c}(z) = q^{\binom{h}{2}} \qbinom{p+h}{h}\qbinom{p+b-1}{b-h} \prod_{j=0}^{p+b-1}(1+zq^j). \]
    
    Taking the coefficient of $z^v$, using the $q$-binomial theorem, \Cref{thm:q-binomial}
    \[
        [z^v]\prod_{j=0}^{p+b-1}(1+zq^j)
        =q^{\binom{v}{2}}\qbinom{p+b}{v},
    \]
    and including the already computed contributions, we obtain
    \[
        \sum_{\pi\mapsto\bar\pi} q^{\dinv(\pi)-\dinv(\bar\pi)}
        = q^{\binom{p+h}{2}} \qbinom{r}{p+h} q^{\binom{h}{2}} \qbinom{p+h}{h}
            \qbinom{p+b-1}{b-h} q^{\binom{v}{2}} \qbinom{p+b}{v} \qbinom{p+b+u-1}{u},
    \]
    where the sum is over preimages with the fixed parameters $b,p,u,v,h$. 
    
    Multiplying by
    $t^{n-r-k-l}t^{l-b}t^{\area(\bar\pi)}q^{\dinv(\bar\pi)}$
    and summing over $\bar\pi$ and all admissible $b,p,u,v,h$ proves the recursion.
\end{proof}

\begin{remark}
    Unfortunately, the argument used in \Cref{prop:combinatorial_recursion} cannot be simplified to isolate the contribution of the $\dinv$ of type \eqref{enum:ndinv_1} between letters $N$ and $V$ or between letters $N$ and $D$ as a $q$-binomial independent of the contribution of the decorated rises, in a similar fashion as in the proof \cite[Theorem~4.11]{DAdderioIraci2023}. Indeed, the relative order between these letters also affects the contributions to the $\dinv$ related to the rise decorations on the first diagonal, and the corresponding terms do not factor nicely. An example of this correlation is shown in \Cref{ex:no-factorization}.
    
    This technical difficulty is the reason why finding the correct statistic $\dinv$ for the Theta conjecture was so challenging. The algebraic identities could not be interpreted combinatorially directly in terms of attacking pairs, but required a more subtle combinatorial analysis.
    
    However, \Cref{lem:TU_formulas} shows that putting together these contributions gives us a nice formula, allowing us to prove the Schröder case of the Theta conjecture.
\end{remark}

\begin{example}
    \label{ex:no-factorization}
    Let $n=4$, $r=2$ and $k=l=d=1$. The coefficient of $\LD_{q,t}(1,1)^{\ast 0, \bullet 0}_{\circ 0}$ in the recursion to compute $\LD_{q,t}(4,2)^{\ast 1, \bullet 1}_{\circ 1}$ is
    \[ \sum_{h=0}^1 q^{\binom{h}{2}} \qbinom{h+1}{h} q^{\binom{h+1}{2}} \qbinom{2}{h+1} \qbinom{2}{1},\]
    obtained by selecting the term with $b=p=v=1$ and $u=0$. Fixing $h=1$ this reduces to 
    \[q(q+1)(q+1).\]
    In this case $\bar\pi$ is the unique path of size $1$ with neither decorations nor marked peaks, and its preimages with these parameters are the four paths in \Cref{fig:path_cont}.
    
    The deleted steps of the first two have the same relative order, with word $W_1=NDN$, and their $\dinv$ values are, respectively $3$ and $1$, so their contribution is $q(q^2+1)$. In a similar way, the deleted steps of the other two have the same relative order $W_2=NND$, and they both have $\dinv$ equal to $2$, contributing $2q^2$ to the coefficient.
    
    Neither of these two contributions can be isolated in the $q$-binomial coefficients of the recursion, so we cannot compute the contributions to the $\dinv$ separately by type of deleted steps, as in \cite[Theorem~4.11]{DAdderioIraci2023}, and a global argument is required.
\end{example}

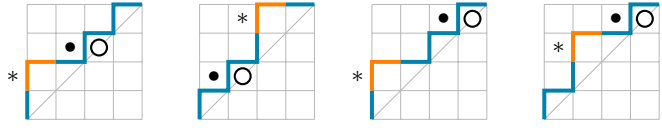
\begin{figure}[htbp]
    \definecolor{recOrange}{RGB}{255,128,0}
    \definecolor{recViolet}{RGB}{108,83,220}
    \definecolor{recGreen}{RGB}{0,128,0}
    \definecolor{recPink}{RGB}{250,110,180}
    \definecolor{recBlue}{RGB}{0,130,170}
    \definecolor{recRed}{RGB}{190,0,0}

    \centering
    \begin{tikzpicture}[
        x=.38cm,
        y=.38cm,
        rec grid/.style={step=1, gray!55, line width=.3pt},
        rec diagonal/.style={gray!55, line width=.4pt},
        rec path/.style={line width=1.5pt, line cap=butt, line join=miter},
        rec sign/.style={inner sep=0pt, font=\small},
        rec peak/.style={line width=.8pt, fill=white}
    ]
        \begin{scope}[shift={(0,0)}]
            \draw[rec grid] (0,0) grid (4,4);
            \draw[rec diagonal] (0,0)--(4,4);
            \draw[rec path, draw=recBlue] (0,0)--(0,1);
            \draw[rec path, draw=recOrange] (0,1)--(0,2)--(1,2);
            \draw[rec path, draw=recBlue]
                (1,2)--(2,2)--(2,3)--(3,3)--(3,4)--(4,4);
            \node[rec sign] at (-.5,1.5) {$\ast$};
            \node[rec sign] at (1.5,2.5) {$\bullet$};
            \draw[rec peak, draw=black] (2.5,2.5) circle[radius=1.05mm];
        \end{scope}

        \begin{scope}[shift={(6,0)}]
            \draw[rec grid] (0,0) grid (4,4);
            \draw[rec diagonal] (0,0)--(4,4);
            \draw[rec path, draw=recBlue]
                (0,0)--(0,1)--(1,1)--(1,2)--(2,2)--(2,3);
            \draw[rec path, draw=recOrange] (2,3)--(2,4)--(3,4);
            \draw[rec path, draw=recBlue] (3,4)--(4,4);
            \node[rec sign] at (1.5,3.5) {$\ast$};
            \node[rec sign] at (.5,1.5) {$\bullet$};
            \draw[rec peak, draw=black] (1.5,1.5) circle[radius=1.05mm];
        \end{scope}

        \begin{scope}[shift={(12,0)}]
            \draw[rec grid] (0,0) grid (4,4);
            \draw[rec diagonal] (0,0)--(4,4);
            \draw[rec path, draw=recBlue] (0,0)--(0,1);
            \draw[rec path, draw=recOrange] (0,1)--(0,2)--(1,2);
            \draw[rec path, draw=recBlue]
                (1,2)--(2,2)--(2,3)--(3,3)--(3,4)--(4,4);
            \node[rec sign] at (-.5,1.5) {$\ast$};
            \node[rec sign] at (2.5,3.5) {$\bullet$};
            \draw[rec peak, draw=black] (3.5,3.5) circle[radius=1.05mm];
        \end{scope}

        \begin{scope}[shift={(18,0)}]
            \draw[rec grid] (0,0) grid (4,4);
            \draw[rec diagonal] (0,0)--(4,4);
            \draw[rec path, draw=recBlue]
                (0,0)--(0,1)--(1,1)--(1,2);
            \draw[rec path, draw=recOrange] (1,2)--(1,3)--(2,3);
            \draw[rec path, draw=recBlue]
                (2,3)--(3,3)--(3,4)--(4,4);
            \node[rec sign] at (.5,2.5) {$\ast$};
            \node[rec sign] at (2.5,3.5) {$\bullet$};
            \draw[rec peak, draw=black] (3.5,3.5) circle[radius=1.05mm];
        \end{scope}
    \end{tikzpicture}
    \caption{The four paths in $\LD(4,2)^{\ast 1, \bullet 1}_{\circ 1}$ with $b=p=v=h=1$ and $u=0$.}
    \label{fig:path_cont}
\end{figure}

\subsection{A detailed example}
\label{ssc:combinatorial-deletion-example}

We take a closer look at the preimages of a path $\bar\pi$ in $\LD(n-r-b, u+v)^{\ast k-v, \bullet l-b}_{\circ d-(r-p)}$ under the map $\pi \mapsto \bar\pi$ defined in the proof of \Cref{prop:combinatorial_recursion}. We will illustrate how the parameters $b,p,u,v,h$ are determined and how they affect the $\dinv$ of the preimages, via a detailed example.

We consider the path $\pi$ in \Cref{fig:combinatorial-deletion}, which is an element of $\LD(18,5)^{\ast 4, \bullet 6}_{\circ 4}$. The path $\bar\pi$ is obtained by deleting the vertical steps starting on the main diagonal and the horizontal steps immediately preceding the next return to the main diagonal. In this case, we have $b=4$, $p=2$, $u=2$, $v=3$, and $h=2$. The path $\bar\pi$ is an element of $\LD(9,5)^{\ast 1, \bullet 2}_{\circ 1}$.

In the deletion map depicted in the figure, red steps and decorations are removed. The other colours identify the five touch-subpaths of $\bar\pi$ and their preimages.
The green and pink subpaths are concatenated inside the same component of $\pi$, between its deleted outer steps.
The letters below the main diagonal record $W=NVVDNDNPN$. Stars decorate rises, bullets decorate valleys, and circles mark peaks.

\begin{figure}[htbp]
    \definecolor{recOrange}{RGB}{255,128,0}
    \definecolor{recViolet}{RGB}{108,83,220}
    \definecolor{recGreen}{RGB}{0,128,0}
    \definecolor{recPink}{RGB}{250,110,180}
    \definecolor{recBlue}{RGB}{0,130,170}
    \definecolor{recRed}{RGB}{190,0,0}

    \centering
    \begin{tikzpicture}[
        x=.38cm,
        y=.38cm,
        rec grid/.style={step=1, gray!55, line width=.3pt},
        rec diagonal/.style={gray!55, line width=.4pt},
        rec path/.style={line width=1.5pt, line cap=butt, line join=miter},
        rec sign/.style={inner sep=0pt, font=\small},
        rec peak/.style={line width=.8pt, fill=white}
    ]
        % Original path
        \draw[rec grid] (0,0) grid (18,18);
        \draw[rec diagonal] (0,0)--(18,18);
        % \draw[rec path, draw=recRed]
        %     (0,0)--(0,3)--(2,3)--(2,4)--(4,4)--(4,5)
        %     --(5,5)--(5,7)--(7,7)--(7,8)--(8,8)--(8,11)
        %     --(10,11)--(10,12)--(11,12)--(11,13)
        %     --(13,13)--(13,14)--(14,14)--(14,16)
        %     --(16,16)--(16,17)--(17,17)--(17,18)--(18,18);

        \draw[rec path,draw=recRed]
            (0,0)--(0,1);

        \draw[rec path,draw=recRed]
            (3,4)--(4,4)--(4,5)
            --(5,5)--(5,6);
            
        \draw[rec path,draw=recRed]
            (6,7)--(7,7)--(7,8)--(8,8)--(8,9);

        \draw[rec path,draw=recRed]
            (12,13)--(13,13)--(13,14)--(14,14)--(14,15);

        \draw[rec path,draw=recRed]
            (15,16)--(16,16)--(16,17)--(17,17)--(17,18)--(18,18);

        % Surviving coloured subpaths
        \draw[rec path,draw=recOrange]
            (0,1)--(0,3)--(2,3)--(2,4)--(3,4);

        \draw[rec path,draw=recViolet]
            (5,6)--(5,7)--(6,7);

        % Green and pink meet at (11,12):
        % use the same sharp junction style as in Fig. pp0.
        \draw[
            rec path,
            draw=recGreen,
            -sharp >,
            sharp angle=45
        ]
            (8,9)--(8,11)--(10,11)--(10,12)--(11,12);

        \draw[
            rec path,
            draw=recPink,
            sharp <-,
            sharp angle=45
        ]
            (11,12)--(11,13)--(12,13);

        \draw[rec path,draw=recBlue]
            (14,15)--(14,16)--(15,16);

        % Decorations that disappear are red; surviving decorations are black.
        \foreach \x/\y in {-.5/1.5,4.5/6.5,13.5/15.5}
            \node[rec sign,text=recRed] at (\x,\y) {$\ast$};

        \node[rec sign] at (-.5,2.5) {$\ast$};

        \foreach \x/\y in {3.5/4.5,4.5/5.5,6.5/7.5,12.5/13.5}
            \node[rec sign,text=recRed] at (\x,\y) {$\bullet$};

        \foreach \x/\y in {1.5/3.5,9.5/11.5}
            \node[rec sign] at (\x,\y) {$\bullet$};

        \foreach \x/\y in {7.5/7.5,13.5/13.5,16.5/16.5}
            \draw[rec peak,draw=recRed]
                (\x,\y) circle[radius=1.05mm];

        \draw[rec peak,draw=black]
            (.5,2.5) circle[radius=1.05mm];

        % Word below the diagonal
        \foreach \j/\lab in {
            0/N, 4/V, 5/V, 7/D, 8/N, 13/D, 14/N, 16/P, 17/N
        }
            \node[font=\scriptsize,text=gray!75!black]
                at ({\j+.45},{\j-.15}) {${\lab}$};

        \node at (9,-1.2) {$\pi$};

        \draw[->,line width=.7pt]
            (18.8,9)--(20.7,9);

        % Reduced path
        \begin{scope}[shift={(21.5,4.5)}]
            \draw[rec grid] (0,0) grid (9,9);
            \draw[rec diagonal] (0,0)--(9,9);

            % Orange: sharp joint at the endpoint (3,3)
            \draw[
                rec path,
                draw=recOrange,
                -sharp >,
                sharp angle=45
            ]
                (0,0)--(0,2)--(2,2)--(2,3)--(3,3);

            % Violet: sharp joint at both ends
            \draw[
                rec path,
                draw=recViolet,
                sharp <-sharp >,
                sharp angle=45
            ]
                (3,3)--(3,4)--(4,4);

            % Green: sharp joint at both ends
            \draw[
                rec path,
                draw=recGreen,
                sharp <-sharp >,
                sharp angle=45
            ]
                (4,4)--(4,6)--(6,6)--(6,7)--(7,7);

            % Pink: sharp joint at both ends
            \draw[
                rec path,
                draw=recPink,
                sharp <-sharp >,
                sharp angle=45
            ]
                (7,7)--(7,8)--(8,8);

            % Blue: sharp joint only at its starting point
            \draw[
                rec path,
                draw=recBlue,
                sharp <-,
                sharp angle=45
            ]
                (8,8)--(8,9)--(9,9);

            % Surviving decorations
            \node[rec sign] at (-.5,1.5) {$\ast$};

            \foreach \x/\y in {1.5/2.5,5.5/6.5}
                \node[rec sign] at (\x,\y) {$\bullet$};

            \draw[rec peak,draw=black]
                (.5,1.5) circle[radius=1.05mm];

            \node at (4.5,-1.2) {$\bar\pi$};
        \end{scope}
    \end{tikzpicture}

    \caption{The deletion map for a path of size $18$. Red steps and decorations are deleted.}
    \label{fig:combinatorial-deletion}
\end{figure}
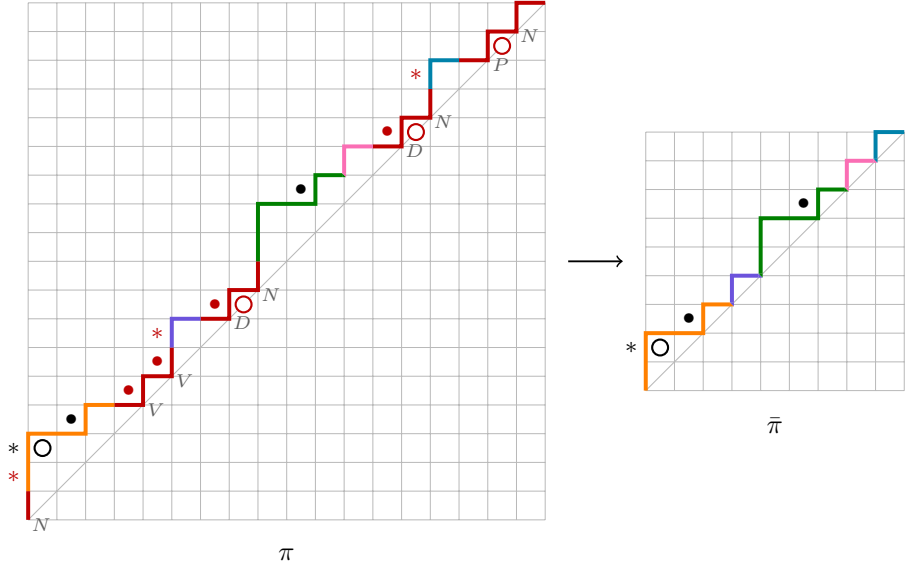

\begin{example}
    \label{ex:combinatorial-deletion}
    Consider the paths in \Cref{fig:combinatorial-deletion}, with north steps indexed from bottom to top. Their area words are
    \[
    \begin{aligned}
    a(\pi)&=(0,1,2,1,0,0,1,0,0,1,2,1,1,0,0,1,0,0),\\
    a(\bar\pi)&=(0,1,0,0,0,1,0,0,0).
    \end{aligned}
    \]
    Their decorations and markings are
    \[
    \begin{aligned}
    \dr(\pi)&=\{2,3,7,16\},
    &\dr(\bar\pi)&=\{2\},\\
    \dv(\pi)&=\{4,5,6,8,12,14\},
    &\dv(\bar\pi)&=\{3,7\},\\
    \mpe(\pi)&=\{3,8,14,17\},
    &\mpe(\bar\pi)&=\{2\}.
    \end{aligned}
    \]
    The star valleys are
    \[
        \dv^{\ast}(\pi)=\{4,5,8\},
        \qquad
        \dv^{\ast}(\bar\pi)=\{3\}.
    \]

    The deleted north steps and their types are
    \[
        \begin{array}{c|ccccccccc}
            \text{index in }\pi & 1 & 5 & 6 & 8 & 9 & 14 & 15 & 17 & 18 \\ \hline
            \text{type}         & N & V & V & D & N & D  & N  & P  & N
        \end{array}
    \]
    The star and nonstar valleys at $4$ and $12$ survive as
    valleys $3$ and $7$ of $\bar\pi$, respectively.  The decorated rise and
    marked peak at $3$ also survive, both at index $2$.

    For the area statistic, we obtain
    \[
        \area(\pi)=11-(1+2+1+1)=6,
        \qquad
        \area(\bar\pi)=2-1=1.
    \]
    Let $D_1,D_2,D_3,D_4$ count the four types of dinv pairs in the order listed in the definition, let $V_0$ count the nonstar valleys, and let $R$ count the rise--valley correction pairs. Then
    \[
        \begin{array}{c|rrrr|rr|r}
                     & D_1 & D_2 & D_3 & D_4 & V_0 & R & \dinv \\ \hline
            \pi      & 28  & 22  & 4   & 10  & 3   & 8 & 53    \\
            \bar\pi  & 14  & 9   & 1   & 4   & 1   & 2 & 25
        \end{array}
    \]
    The correction pairs counted by $R$ are
    \[
        \begin{aligned}
            \pi:\quad
                &(2,5),(2,6),(2,8),(2,14),\\
                &(3,4),(3,12),(7,8),(7,14),\\
            \bar\pi:\quad
                &(2,3),(2,7).
        \end{aligned}
    \]

    The parameters are
    \[
        (n,r,k,l,d)=(18,5,4,6,4),
        \qquad
        (b,p,u,v,h)=(4,2,2,3,2).
    \]
    In particular, $\bar\pi$ has size $9$ and $u+v=5$ touches.  The area
    difference agrees with the recursion:
    \[
        \area(\pi)-\area(\bar\pi)=6-1=5=n-r-b-k.
    \]

    The word is $W=(NVVD)(ND)(NP)(N)$, with associated parameters $s = p+h = 4$ and 
    \[
        \begin{gathered}
            W=(NVVD)(ND)(NP)(N), \qquad s=4,\\
            (b_0,b_1,b_2,b_3,b_4)=(0,0,0,1,0),\\
            (c_1,c_2,c_3,c_4)=(2,0,0,0),\qquad
            (\varepsilon_1,\varepsilon_2,\varepsilon_3,\varepsilon_4)
                =(1,1,0,0).
        \end{gathered}
    \]
    The $\dinv$ contribution of the steps of $W$ is
    \[
        \binom{4}{2}
        +\sum_{i=1}^{4}(i-1)(c_i+\varepsilon_i)
        +\#\{N\text{ before }P\}
        =6+1+3=10.
    \]
    Indexing the available positions from right to left by
    $j=0,\ldots,5$, the insertion data are
    \[
        \begin{aligned}
            (u_0,\ldots,u_5)&=(0,0,2,0,0,0),\\
            (\delta_0,\ldots,\delta_5)&=(0,1,0,1,0,1).
        \end{aligned}
    \]
    The entry $u_2=2$ records the green and pink touch-subpaths in the same
    position.  Their total ordinary contribution is
    \[
        \sum_j ju_j=2\cdot2=4.
    \]
    The three star insertions realize the three cases in the proof.  The
    insertion at $2$ makes the unmarked valley at $5$ a star valley and has exponent
    $0+2(4-1)+2=8$; the insertion at $7$ makes the marked valley at $8$
    a star valley and has exponent $2(4-1)-1=5$; and the insertion at $16$ creates
    no star valley and has exponent $4-3=1$.  These exponents
    include the placement weights of the corresponding subpaths.

    Consequently,
    \[
        \dinv(\pi)-\dinv(\bar\pi)=10+4+(8+5+1)=28,
    \]
    in agreement with the direct calculation $53-25=28$.
    Therefore,
    \[
        q^{\dinv(\pi)}t^{\area(\pi)}
        =q^{28}t^5q^{\dinv(\bar\pi)}t^{\area(\bar\pi)}
        =q^{53}t^6.
    \]
    For comparison, the recursion gives the following relative weight for
    all preimages of this fixed $\bar\pi$ with these parameters:
    \[ t^5 q^4 \qbinom{5}{2} \qbinom{8}{4} \qbinom{6}{3} \qbinom{7}{2}. \]
\end{example}

\section{Artificial intelligence}
\label{sec:ai}

Artificial intelligence has been used in this paper to assist both with the mathematical research and with the writing of the paper. In this section we describe how.

\subsection{Scaffolding}

The statistic $\dinv$ was discovered with the help of Aristotle \cite{Achim2025Aristotle}, available for free at \url{https://aristotle.harmonic.fun}. Aristotle was prompted with some literature \cites{DAdderioIraci2023,IraciNadeauVandenWyngaerd2024Smirnov,DIIP2026LeavingTheHall}, the key features expected from the statistic (nondecorated steps, decorated valleys, star valleys), a Python/SageMath implementation of the symmetric function side of the conjecture, and the instructions to find a statistic that would match the symmetric function side for $n-k-l=1$ at first, and in general later.

Aristotle searched for a statistic of type
\[ \dinv(\pi) = \sum_{i<j} w(cls_i, cls_j, \delta_{ij}, \epsilon_{ij})  +  \sum_i u(cls_i), \]
where $cls_i$ encodes the type of the $i\th$ step (not decorated, decorated rise, decorated valley, star valley), $\delta_{ij}$ encodes the relative position of the $i\th$ and $j\th$ steps (same diagonal, adjacent diagonals, far away), $\epsilon_{ij}$ encodes the inequality between $w_i$ and $w_j$, with constraints given by the known distributions. For objects of size up to $6$, this gives a linear system with $120$ unknowns; matching the symmetric function side for $n \leq 6$ produces a system of $1609$ equations with rank $110$. The solution space has dimension $10$, and gives $5$ different solutions with coefficients in $\{-1,0,1\}$.

The solution presented in this paper is the smallest in norm. The solutions match for $n \leq 5$ but differ for $n = 6$; other solutions were discarded because they did not match the symmetric function side for $n=7$; the one we present in this paper was then validated up to $n=10$. 

Notably, GPT-6 Astra Ultra failed on the same task without human assistance\footnote{GPT-6 Astra Ultra also suggested to remove this line during the revising process.}.

\subsection{Editing}

The authors subsequently used AI-assisted tools during revision for proof checking, producing examples, and language editing. The authors carefully checked, revised, and rewrote all resulting arguments, and take full responsibility for the final content.

\bibliographystyle{amsalpha}
\bibliography{references}

\end{document}